\documentclass[twoside,11pt]{amsart}

\usepackage{amsmath,latexsym,amssymb, times, enumerate, xypic, stmaryrd, hyperref}
\usepackage{amscd, color}
\input amssym.def
\input amssym
\input xypic
\input xy
\xyoption{all}
\def\demo{\noindent{\bf Proof. }}
\def\sqr#1#2{{\vcenter{\hrule height.#2pt
        \hbox{\vrule width.#2pt height#1pt \kern#1pt
                \vrule width.#2pt}
        \hrule height.#2pt}}}
\def\square{\mathchoice\sqr64\sqr64\sqr{4}3\sqr{3}3}
\def\QED{\hfill$\square$}

\def\tratto{\mbox{\rule{2mm}{.2mm}$\;\!$}}

\def\m{{\mathfrak m}}
\def\n{{\mathfrak n}}

\def\p{{\mathfrak p}}

\newtheorem{Theorem}{Theorem}[section]
\newtheorem{Claim}{Claim}

\newtheorem{Lemma}[Theorem]{Lemma}
\newtheorem{Corollary}[Theorem]{Corollary}
\newtheorem{Proposition}[Theorem]{Proposition}

\newtheorem{Notation and Discussion}[Theorem]{Notation and Discussion}
\newtheorem{Assumptions and Discussion}[Theorem]{Assumptions and Discussion}
\newtheorem{Setting}[Theorem]{Setting}

\newtheorem{Example}[Theorem]{Example}
\newtheorem{Definition}[Theorem]{Definition}

\begin{document}

\baselineskip=16pt

\title[Generalized stretched ideals and  Sally's Conjecture]
{\Large\bf Generalized stretched ideals and Sally's Conjecture}

\author[Paolo Mantero]{Paolo Mantero$^{1}$}
\address{University of California, Riverside \\ Department of Mathematics \\
 Riverside, CA 92521}
\email{mantero@math.ucr.edu\newline
\indent{\it URL:} \href{http://math.ucr.edu/~mantero/}{\tt http://math.ucr.edu/$\sim$mantero/}}

\author{Yu Xie}
\address{Department of Mathematics and Statistics, Penn State Altoona, Altoona, PA 16601}
\email{yzx1@psu.edu}

\thanks{AMS 2010 {\em Mathematics Subject Classification}.
Primary 13A30; Secondary 13H15,
13C14, 13C15.
}
\thanks{$^1$ P. Mantero gratefully acknowledges the support of an AMS-Simons Travel Grant}

\vspace{-0.1in}

\begin{abstract}
We introduce the concept of $j$-stretched ideals in a Noetherian local ring. This notion 
generalizes to arbitrary ideals the classical notion of stretched $\m$-primary ideals of Sally and Rossi-Valla,
as well as the concept of ideals of  minimal and almost minimal $j$-multiplicity introduced 
 by Polini-Xie. One of our main theorems states that, for a $j$-stretched ideal, the associated graded ring is Cohen-Macaulay if and only if two classical invariants of the ideal, the reduction number and the index of nilpotency, are equal.
Our second main theorem, presenting numerical conditions which ensure the almost Cohen-Macaulayness of
the associated graded ring of a $j$-stretched ideal, provides   a generalized version of Sally's conjecture.
This work, which also holds for modules,
 unifies the approaches of Rossi-Valla and  Polini-Xie  and generalizes  simultaneously results on the  Cohen-Macaulayness or  almost Cohen-Macaulayness of the associated graded  module by several authors, including Sally, Rossi-Valla, Wang, Elias,  Corso-Polini-Vaz Pinto, Huckaba, Marley and Polini-Xie.
\end{abstract}

\maketitle

\vspace{-0.2in}

\section{Introduction}

Given a Noetherian local ring $(R,\m)$ and an ideal $I$ of $R$, it is well-known that the associated graded ring ${\rm gr}_I(R)=\oplus_{n=0}^{\infty}I^n/I^{n+1}$ encodes algebraic and geometric properties of $I$. Indeed, ${\rm Proj}({\rm gr}_I(R))$ is the exceptional fiber of the blow-up of ${\rm Spec}(R)$ along the subvariety $V(I)$.
Strong efforts have been given in the last thirty years to detect  conditions on $R$ and $I$ which guarantee that ${\rm gr}_I(R)$ has  sufficiently high depth (more precisely, ${\rm gr}_I(R)$ being Cohen-Macaulay or almost Cohen-Macaulay),
 due to the reason that  high depth of the associated graded ring  forces the vanishing of its cohomology groups
and thereby allows one to  compute, or bound, relevant numerical invariants
such as the Castelnuovo-Mumford regularity or the number and degrees of the defining equations of the blow-up
(see for instance, \cite{JK} and \cite{JU}).

The classical method, originated from the pioneering work of Sally,
studies the interplay between the Hilbert coefficients of an $\m$-primary ideal  and
the depth of the associated graded ring. The idea is that extremal values of the Hilbert coefficients yield high depth of the
associated graded ring and, conversely, good depth properties encode all the information about the Hilbert function.

In~1967, Abhyankar proved that the  multiplicity   of a
$d$-dimensional Cohen-Macaulay local ring $(R, \m)$  can be written as
$e_0(\m)=\mu(\m)-d+K$ for some integer $K\geq 1$,
where $\mu(\m)$ is the embedding dimension of $R$ \cite{AB}. 
Since then, rings for which $e_0(\m)=\mu(\m)-d+1$ (respectively, $e_0(\m)=\mu(\m)-d+2$) have been
called {\it rings of minimal multiplicity} (respectively, {\it rings of almost minimal multiplicity}).
These notions were extended by Sally to stretched Cohen-Macaulay local rings by requiring
an Artinian reduction $R/J$, where $J$ is a minimal reduction of $\m$, to be stretched, i.e.,
 the ideal $(\m/J)^2$ is  a principal ideal
(see  \cite{S3} and \cite{RV3}). Sally studied   the Cohen-Macaulay and almost Cohen-Macaulay property
of the associated
graded ring ${\rm gr}_{\m}(R)$ for those classes of rings.
She proved that  ${\rm gr}_\m(R)$ is always Cohen-Macaulay if $R$ has minimal
multiplicity \cite{S1}. 
Unfortunately, for arbitrary Cohen-Macaulay local rings of almost minimal
multiplicity (as well as stretched Cohen-Macaulay local rings), 
 the Cohen-Macaulay property of ${\rm gr}_\m(R)$ fails to
hold \cite{S4}.
However, Sally conjectured that if $R$ has almost minimal multiplicity then ${\rm gr}_\m(R)$ is almost Cohen-Macaulay.
This conjecture was proved thirteen years later
by Rossi and  Valla
\cite{RV1}, and, independently, by  Wang~\cite{W}.
Later, in 2001, Rossi and Valla extended the notion of stretched Cohen-Macaulay local rings of Sally to stretched $\m$-primary ideals, and proved an extended version of Sally's conjecture by giving  conditions for  the associated graded rings of stretched $\m$-primary ideals to be almost Cohen-Macaulay \cite{RV3}.

During the last twenty years, another method has also been developed to study the depth of the associated
graded rings of general ideals 
(see \cite{HH1},   \cite{T},  \cite{GH}, \cite{GY1},  \cite{GYN}, \cite{JU}, \cite{Gh},   \cite{AGH},     and related papers).
Essentially, this method requires  the ideal $I$ to have certain residual intersection properties (automatically satisfied if  $I$ is $\m$-primary) and sufficiently many powers of $I$ to have high  depth, where the number of powers of $I$ required to have high depth depends on the reduction number of~$I$. Since the depth   drops dramatically for higher powers of $I$, this method works well if $I$ has ``relatively small" reduction number.

 Recently, Polini and Xie \cite{PX1} proved Sally's conjecture for a class of ideals that are not necessarily  $\m$-primary
 by combining
 the techniques of $\m$-primary ideals with tools from  residual intersections.
 They extended the notions of  minimal and almost minimal multiplicity to arbitrary ideals by defining the concepts of minimal and almost minimal $j$-multiplicity, and proved that,  under certain residual assumptions,  the associated graded ring
is Cohen-Macaulay (respectively, almost Cohen-Macaulay) for  ideals  having minimal
$j$-multiplicity (respectively, almost minimal $j$-multiplicity).\\
\\
In the present paper,  we propose  a more general numerical condition on $I$ that extends
the classical stretched  $\m$-primary ideals defined by Sally, Rossi and Valla,
as well as the minimal and almost minimal $j$-multiplicity introduced by Polini and Xie.
Let $(R, \m)$ be a Noetherian local ring of dimension $d$ with infinite residue field
(we can enlarge the residue field to be infinite
 by replacing $R$ by $R(z)=R[z]_{\m R[z]}$, where $z$ is a variable over $R$).
Let $I$ be an $R$-ideal of maximal analytic spread. Recall that the quotient ring of $R$ modulo a general $d-1$-geometric residual  intersection of $I$ is a $1$-dimensional Noetherian local ring and the ideal generated by the image of  $I$ in this quotient ring is  primary to its maximal ideal (thus it allows us to reduce to the setting of the classical $\m$-primary case). Roughly speaking,
 the ideal    $I$ is {\it $j$-stretched} if  it generates a stretched $\m$-primary ideal
 (in the sense of Rossi and Valla) after reducing to this   $1$-dimensional Noetherian  local ring.
Since $j$-stretched ideals are not necessarily $\m$-primary,  to study them,
 we adopt  the tools of general elements, residual intersection theory
  (a generalization of linkage), and  the notion of  $j$-multiplicity
   (introduced by Archilles and Manaresi as a higher dimensional version of the Hilbert multiplicity \cite{AM}). We refer to Section 2 in the following for a more detailed elaboration of $j$-stretched ideals.

  One of the most important features of $\m$-primary ideals $I$  comes from the fact that they have finite colength
  $\lambda (R/I)$, which makes many tools and computations applicable.
  When $I$ is arbitrary,  one would like to reduce to the case of finite colength
by factoring out a sequence of  elements. But the problem is that the colength depends on
 the choice of a sequence of elements. To overcome this difficult, we develop a ``Specialization Lemma"
  (see Lemma \ref{specializ} in Section 3) stating that,
if we choose a sequence of
general elements, we will have a fixed colength. Moreover, if $R$ is equicharacteristic, general specializations yield the smallest colength. We apply this lemma  to study  the index of nilpotency and the stretchedness property. For instance, we generalize to non $\m$-primary ideals $I$ a proposition proved by Fouli \cite[Proposition~5.3.3]{F}, stating
that, over an equicharacteristic  Cohen-Macaulay local ring, the index of nilpotency of $I$ does not depend on the general minimal reduction, and general minimal reductions always achieve the largest possible index of nilpotency.  We also answer a question of Sally (see \cite{S3}) asking: to what extent does the classical notion of stretchedness depend on  minimal reductions? As a consequence of  Lemma 3.1, one obtains the answer that the stretchedness property does not depend on the choice of a general minimal reductions. We remark here that Lemma 3.1 may be of independent interest to the reader, 
as it can also be interpreted as an upper-semicontinuity result of lengths.

We now state our main theorems. For any $j$-stretched ideal $I$ with certain residual intersection properties
(automatically satisfied if $I$ is $\m$-primary),
we prove in Theorem \ref{CM} that the associated graded ring ${\rm gr}_I(R)$ is Cohen-Macaulay if and only if the reduction number of $I$ and its index of nilpotency  coincide.
The second main result, Theorem \ref{2},  provides a sufficient condition for the associated graded rings of
$j$-stretched ideals  to be almost Cohen-Macaulay   and is a  generalized version of Sally's conjecture.
 Our criteria are purely numerical and could be applied to ideals with arbitrarily large reduction numbers. Indeed, we provide  a class of $j$-stretched ideals having arbitrarily  large reduction number such that the Cohen-Macaulay property of the associated graded ring follows from our main theorem, but from no previous result in the literature (see Example \ref{ex1} in Section 4).\\
\\
The structure of the paper is the following: In Section 2,  we define the concept of $j$-stretched ideals and recall definitions of residual intersections.
Section 3 is rather technical and includes the Specialization Lemma (Lemma \ref{specializ}) as well as several results on the structure of $j$-stretched ideals. Section 4 contains our two main theorems, giving numerical characterizations of the  Cohen-Macaulayness and almost Cohen-Macaulayness of the associated graded rings of $j$-stretched ideals (Theorem \ref{CM} and Theorem \ref{2}).  Among the applications of these theorems, we recover the main results of \cite{PX1} and \cite{RV3}, and prove, under additional assumptions, that the associated graded rings of ideals having almost-almost minimal $j$-multiplicity are almost Cohen-Macaulay (Corollary \ref{almalm}).

Finally, in Section 5, we prove the non-trivial fact that $j$-stretched ideals do generalize stretched $\m$-primary ideals (Theorem \ref{AN} and Corollary \ref{mprim}). Although in general these two notions are different, we provide a sufficient condition for them to coincide (Proposition \ref{equiv}). As an application, we answer  a question raised by Sally (Corollary \ref{Sally}).

For the sake of clarity, we will only focus on the case of associated graded rings ${\rm gr}_I(R)$, although all the definitions and results can be extended and proved for associated graded modules ${\rm gr}_I(M)$, where $M$ is a finite module over $R$.

\section{The Main Definitions}

In this section we  fix the notation, introduce $j$-stretched ideals and recall some definitions and facts from residual intersection theory.

Throughout this paper, we always assume that $(R, \m, k)$ is a Noetherian local ring of dimension $d$ with maximal ideal $\m$ and infinite residue field  $k=R/\m$ (possibly, after
enlarging the residue field $k$).
\begin{itemize}
\item The {\em associated graded ring} of an $R$-ideal $I$ is defined as $G={\rm gr}_I(R)=\oplus_{n=0}^{\infty}I^n/I^{n+1}$.
\item An ideal $J\subseteq I$ is called a {\em reduction} of $I$  if there
exists a non-negative integer $r$ such that $I^{r+1}=JI^r$.  The least  $r$ such that $I^{r+1}=JI^r$ is denoted by $r_J(I)$, and called the {\em reduction number of $I$ with respect to $J$}.
\item A reduction is called {\it minimal} if it is minimal with respect to inclusion.
\item  The {\it reduction number} $r(I)$ of $I$ is defined as
 ${\rm min}\{r_J(I)\,|\, J \, {\rm a\, minimal\, reduction\, of \,} I\}$.
 \item Finally, since $|k|=\infty$, minimal reductions of $I$ always exist, and every minimal reduction of $I$ can be minimally generated by the same number of generators, $\ell(I)$, dubbed the {\em analytic spread}  of $I$. Since the inequality $\ell(I)\leq d={\rm dim}\,R$ always holds, one says that $I$ has {\em maximal analytic spread} if $\ell(I)=d$.
\end{itemize}

 Write $I =(a_1, \ldots,a_s)$ and  $x_i =\sum_{j=1}^s \lambda_{ij}a_j$ for $i=1,\dots,t$ and $(\lambda_{ij})\in R^{ts}$.
 The elements $x_1,\dots,x_t$ are {\em general} in $I$ if there exists a Zarisky dense open subset $U$ of
 $k^{ts}$ such that  $(\overline{\lambda_{ij}}) \in U$, where $^{^{\tratto}}$   denotes images in the residue field $k$. 
The relevance of this notion in our analysis comes from the following facts:\begin{itemize}
\item[(a)] General elements in $I$ always form a superficial sequence for $I$ (\cite[Corollary~2.5]{X});
\item[(b)]  If $t=\ell(I)$ then a sequence $x_1,\dots,x_t$ of  general elements in $I$ forms a minimal reduction of $I$ with reduction number  $r(I)$  (see for instance \cite[Corollary~2.2]{T2});
\item[(c)]  One can use general elements to compute the $j$-multiplicity of the ideal $I$ (\cite[Proposition~2.1]{PX1}).
\end{itemize} 

{\bf Notation.} From now on, we assume $I$ has maximal analytic spread  $\ell(I)=d$,
and $J$ is a {\it general} minimal reduction of $I$, i.e.,
$J=(x_1, \ldots, x_d)$, where $x_1, \ldots, x_d$ are
$d$ general elements in $I$.

We write
$\overline{R}=R/J_{d-1}: I^{\infty}$, where $J_{d-1}: I^{\infty}=\{b\in R\,|\, \exists\, \delta>0 \,{\rm such\, that}\, b \cdot I^{\delta}\subseteq J_{d-1}\}$ and $J_{d-1}=(x_1,\ldots,x_{d-1})$. We   use
 $^{\tratto}$  to denote images in the quotient ring  $\overline{R}$.
 
Note that $\overline{R}\neq 0$ if and only if $\ell(I)=d$ \cite{NU}.
Indeed in this case $\overline{R}$ is an 1-dimensional  Cohen-Macaulay local ring
and $\overline{I}$ is primary to the maximal ideal $\overline{\m}$.

 Therefore, one can define  the  Hilbert function of $I$ on $\overline{R}$:
$$
HF_{I,\,\overline{R}}(n)=\lambda (\overline{I^n}/\overline{I^{n+1}}), \,\,\,{\rm for}\,\,n\geq 0,
$$
which is independent of a choice of  the general minimal reduction $J$ (by Lemma \ref{specializ} in Section 3, or see \cite{PX2}).
The $j$-multiplicity of $I$ is computed as follows (see for instance \cite[Proposition~2.1]{PX1})
$$
j(I)=e(I, \overline{R})=\lambda(\overline{R}/x_d \overline{R})=\lambda(\overline{I}/x_d\overline{I})=\lambda(\overline{I}/\overline{I^2})+\lambda(\overline{I^2}/x_d\overline{I}).
$$

We are now ready to give the definition of $j$-{\rm stretched} ideals.
\begin{Definition}\label{Def}
Let $R$,  $I$ and $J$ be the same  as above.
We say that $I$ is $j$-{\rm stretched}  if
$$\lambda(\overline{I^2}/x_d\overline{I}+\overline{I^3})\leq 1.$$
\end{Definition}
Observe that if $I$ is a $j$-stretched ideal  then the Artinian reduction  $\overline{R}/(\overline{x_d})$ possesses a stretched   Hilbert function with respect to $I$, i.e., 
  $$H_{\overline{I}/(\overline{x_d})}(2)=\lambda(\overline{I^2}/(\overline{x_d})\cap \overline{I}^2+\overline{I^3})\leq \lambda(\overline{I^2}/x_d\overline{I}+\overline{I^3})
\leq 1.$$
Furthermore, if   $I$ has {\it minimal $j$-multiplicity} (respectively, {\it almost minimal $j$-multiplicity}), i.e.,   the length $\lambda(\overline{I^2}/x_d\overline{I})=0$ (respectively, $\lambda(\overline{I^2}/x_d\overline{I})\leq 1$) (see \cite{PX1}), then it
is easy to see that $I$ is $j$-stretched; hence
the notion of  $j$-stretched ideals includes ideals having minimal or
almost minimal $j$-multiplicity.
In particular, every $\m$-primary ideal having  minimal or almost minimal multiplicity  is $j$-stretched.
We will see in Section 5 that $j$-stretched ideals also generalize stretched $\m$-primary ideals (Corollary \ref{mprim}).

The property of $j$-stretchedness is preserved under faithfully flat ring extensions. Indeed  let  $(S,\n)$ be a Noetherian local ring that is  flat over  $R$ with $\m S=\n$. If $I$ is $j$-stretched then $IS$ is a $j$-stretched ideal of $S$.
Therefore
the property of being $j$-stretched still holds after passing to the completion of $R$,
or  enlarging the residue field. 
\medskip

We now recall some definitions and facts from the theory of residual intersections (see for instance \cite{U},
 \cite{JU} and \cite{PX1}), which will be used frequently in the rest of the paper.
 \begin{itemize}
 \item An ideal $I$ has the {\it   $G_{t}$ condition} if $I_p$ can be generated by $i$ elements for every
  $\p \in V(I)$ with ${\rm dim}\,R_{\p} = i < t$.
 \item Let $H_{t}=(x_1, \ldots, x_t)$, where $x_1,\ldots,x_t$ are elements in $I$. Define $H_{t}: I=
 \{b\in R\,|\, b\cdot I \subseteq H_t\}.$
  One says that  $H_t:I$ is a {\it  $t$-residual intersection } of $I$
 if $I_\p=(x_1, \ldots, x_{t})_\p$ for every $\p\in {\rm Spec}(R)$ with ${\rm dim}\,R_\p\leq t-1$.
 \item A  $t$-residual intersection $H_t:I$ is called a {\it geometric $t$-residual intersection} of $I$
 if, in addition,  $I_\p=(x_1, \ldots, x_{t})_\p$ for every
 $\p\in V(I)$ with ${\rm dim}\,R_{\p}\leq t$.
 \item It is well-known that, if $I$  satisfies the  $G_t$ condition,
 then for general elements $x_1, \ldots,  x_t$ in $I$ and each $0\leq i< t$,
 the ideal $H_i : I$ is a geometric $i$-residual intersection of $I$,  and $H_t : I$ is a $t$-residual intersection of $I$ (see \cite{U} and \cite[Lemma~3.1]{PX1}).
 \item Finally, let $R$ be   Cohen-Macaulay, the ideal $I$ has the {\it Artin-Nagata property} ${AN^-_t}$ if, for every $0\leq i\leq t$ and every geometric $i$-residual intersection $H_i: I$ of $I$, one has that $R/H_i:I$ is   Cohen-Macaulay \cite{U}.
  \end{itemize}
Assume $R$ is Cohen-Macaulay.  We now list  a
few classes of ideals satisfying the above residual properties.
\begin{itemize}
\item[($\star$)]  The properties $G_d$  and $AN_{d-2}^-$ are automatically satisfied by any $\m$-primary ideal of $R$.
\item[($\star$)] Assume ${\rm dim}(R/I)=1$. Then the property $AN_{d-2} ^-$ is trivially satisfied by $I$. Furthermore, $I$ has the $G_d$ condition if and only if $I$ is generically a complete intersection.
\item[($\star$)] Recall that $I$ is {\it strongly Cohen-Macaulay} if all of the Koszul homology modules with respect to a
generating set of $I$ are Cohen-Macaulay $R/I$-modules. The property $AN_{d-2}^-$ is satisfied by any strongly
Cohen-Macaulay ideal $I$ which satisfies  the $G_d$ condition
     \cite{Hu2}. Examples of strongly Cohen-Macaulay ideals are complete intersections and, if $R$ is Gorenstein, any {\it licci} ideal $I$,  meaning that $I$ is in the linkage class of a complete intersection, which generalizes the classes of perfect ideals of
         grade two and Gorenstein ideals of grade three \cite{Hu}.
\item[($\star$)] Assume $R$ is Gorenstein. Then by linkage theory, the property $AN^-_{d-2}$ is satisfied by any Cohen-Macaulay ideal $I$ with ${\rm dim}(R/I)=2$, and, more generally, by any licci ideal $I$ (for instance, this follows by the above, the facts that the deformation of a licci ideal is licci, and any licci ideal $I$ has a deformation that has the $G_d$ property, and \cite[Lemma~1.13]{U}).
\end{itemize}

We now provide examples  of $j$-stretched  ideals in  Noetherian  local rings.
\begin{Example}\label{r}
Fix $r\geq 1$. Let $R=\mathbb{C}\llbracket{x,y,z}\rrbracket/(x,y)\cap (x^{r+1},z)=\mathbb{C}\llbracket{x,y,z}\rrbracket/(x^{r+1},xz,yz)$ and $I=(x,y)$. Then $R$ is
an $1$-dimensional Cohen-Macaulay local ring and $I$ is a Cohen-Macaulay prime ideal that has $\ell(I)=1$,  $G_1$ condition and $AN_{d-2}^-$ (automatically satisfied since $d=1$). Furthermore,  $I$  is $j$-stretched with reduction number $r$. If $r> 2$ then $I$ does not have almost minimal $j$-multiplicity.
\end{Example}
\begin{proof}
It is easy to see that $R$ is
an $1$-dimensional Cohen-Macaulay local ring and $I$ is a Cohen-Macaulay prime ideal
 that has $\ell(I)=1$,  $G_1$ condition and $AN_{d-2}^-$. We only need to show that $I$ is $j$-stretched with reduction number $r$.
 First notice that $\overline{R}=R/0: I^{\infty}=R/0:I=R/(x^{r+1}, z)\cong k\llbracket{x,y}\rrbracket/ (x^{r+1})$, and use $^{^{\tratto}}$ to  denote images in the quotient ring  $\overline{R}$.
 Let $f=\alpha x+\beta y$ be a general element in $I$. Then $J=(f)$ is a minimal reduction of $I$, hence
  $\beta\neq 0$ since otherwise $I/(f)$ can not have finite length. Replacing $y$ by $f$,
we may assume that $f=y$ and  the length
$$
\lambda(\overline{I^2}/f\overline{I}+\overline{I^{3}})=\lambda [(x,y)^2/(y(x, y)+(x, y)^3 +(x^{r+1}))]=\lambda [(x,y)^2/(y^2,xy,x^{r+1})]\leq 1,
$$
proving the $j$-stretchedness of $I$.  Notice that $y(x, y)^{r-1}+ (x^{r+1})\subsetneq (x,y)^{r}$ and $y(x, y)^{r}+ (x^{r+1})= (x,y)^{r+1}$, hence $r_{(\overline{y})} (\overline{I})=r$.\,
Since $ [0:I]\cap I=0$,  the equality
 $(\overline{y})\overline{I}^r=\overline{I}^{r+1}$  implies that
$$
I^{r+1}\subseteq yI^{r} + [0:I]\cap I^{r+1}=y I^{r},
$$
which gives $r_{(y)}(I)\leq r$. Our desired result follows since
$r_{(y)}(I)\geq r_{(\overline{y})} (\overline{I})= r$.

Finally, since $
\lambda(\overline{I^2}/f\overline{I})=\lambda [(x,y)^2/(y(x, y) +(x^{r+1}))]=r-1,
$
the ideal  $I$ does not have almost minimal $j$-multiplicity if $r> 2$.
\end{proof}

\begin{Example}\label{r2}
Fix $r\geq 1$. Let $R=k\llbracket{x,y,z}\rrbracket/(x^r-yz,y^r-xz, xyz)\cap (x^{r+1}-y^{r+1},z)$ and $I=(x,y)$. Then $R$ is an
$1$-dimensional Noetherian local ring (not Cohen-Macaulay) and $I$ is a Cohen-Macaulay prime ideal that has $\ell(I)=1$, $G_1$ condition and $AN_{d-2}^-$ (automatically satisfied since $d=1$). Furthermore,  $I$ 
is $j$-stretched with reduction number $r$. Write $\overline{R}=R/0:I^{\infty}=
R/(x^{r+1}-y^{r+1},z)\cong k\llbracket{x,y}\rrbracket/ (x^{r+1}-y^{r+1})$. One has $\lambda(\overline{I^t}/J\overline{I^{t-1}}+\overline{I^{t+1}})=1$ for all $2\leq t\leq r$ and 
 a general minimal reduction $J$ of $I$. This implies that $I$ does not have almost minimal $j$-multiplicity if $r> 2$.
\end{Example}

We now exhibit monomial ideals and ideals of points in $\mathbb P^N$ that are $j$-stretched.
\begin{Example}\label{pts}
Assume $I$ is the defining ideal of either  $($i$)$  a set of $n=6$ general points in $\mathbb P^2$, or  $($ii$)$  a set of $n=4$ or $n=5$ general points in $\mathbb P^3$. Then $I$ is a $j$-stretched Cohen-Macaulay ideal which is generated in a single degree, has  $\ell(I)=1$,  $G_1$ condition  and $AN^-_{-1}$.
\end{Example}

\begin{Example}\label{mon1}
Let $I$ be either the ideal $(a^2b^2,a^2c^2,abc^2,b^3c)$ or $(a^3,a^2b,b^2c,ac^2)$ in $R=k[a,b,c]$. Then $I$ is a height 2 ideal that is not unmixed $($indeed, the maximal ideal is an associated prime ideal of $I$). Computations show that $\lambda(\overline{I^t}/x_3\overline{I^{t-1}}+\overline{I^{t+1}})=1$ for $2\leq t\leq 4$, where $\overline{R}=R/(x_1, x_2):I^{\infty}$,  $x_1, x_2$ and $x_3$ are  general elements in $I$. One has that $\ell(I)=3$ and $I$ is $j$-stretched.
Since ${\rm dim}(R/I)=1$, the property $AN^-_{1}$ is automatically satisfied. Moreover, the second ideal has the $G_3$ condition  because  $I$ is generically a complete intersection.
\end{Example}

\begin{Example}
Let  $I=(a^2b^2,a^2c^2,abc^2,b^2c^2,a^2bc)\subseteq R=k[a,b,c]$. Then  $I$
 is a Cohen-Macaulay ideal that is generated in a single degree with $AN^-_{1}$.
 The equality
$\lambda(\overline{I^2}/x_3\overline{I}+\overline{I^3})=1$, where $\overline{R}=R/(x_1, x_2):I^{\infty}$,  $x_1, x_2$ and $x_3$ are  general elements in $I$,  implies  that $\ell(I)=3$ and $I$ is $j$-stretched.
\end{Example}

\section{Structure of $j$-stretched ideals}

In this section we introduce techniques to study the structure of $j$-stretched ideals. These technical results will be employed in the next section to prove our main theorems.
We start  with the proof of the Specialization Lemma (Lemma \ref{specializ}).  To state it, we need to recall the notion of specialization of modules, as introduced by Nhi and Trung \cite{NT}.

Let $S=R[\underline{z}]$, where  $ \underline{z}=z_1, \ldots,  z_{t}$ are variables  over the Noetherian local ring $(R,\m, k)$ (recall   $k$ is infinite and $d={\rm dim}\,R$). Let $M^{\prime}$ be a finite $S$-module.
 Let $\phi:
S^f \rightarrow S^g \rightarrow  0$ be a finite free presentation of
$M^{\prime}$ and let $A=(a_{ij}[\underline{z}])$ be a matrix representation of $\phi$.
For any vector $\underline{\alpha}=(\alpha_1, \ldots, \alpha_t)\in R^t$,
let $A_{\underline{\alpha}}:=(a_{ij}[\underline{\alpha}])$ and
$\phi_{\underline{\alpha}}: R^f\rightarrow R^g\rightarrow  0$ be
the corresponding map defined by $A_{\underline{\alpha}}$.
One says that $\phi_{\underline{\alpha}}$ is a {\it specialization} of $\phi$.
A {\it specialization} of $M^{\prime}$ is defined to be
$M^{\prime}_{\underline{\alpha}}:={\rm Coker}(\phi_{\underline{\alpha}})$.
By \cite{NT}, $M^{\prime}_{\underline{\alpha}}$ does not depend
 on (up to isomorphisms) the choice of $\phi$ and $A$.
 The vector $\underline{\alpha}\in R^t$ is said to
 be {\it general} (equivalently,
 the specialization $M^{\prime}_{\underline{\alpha}}$ is {\it general})
 if the image $\overline{\underline{\alpha}}=(\overline{\alpha}_1, \ldots, \overline{\alpha}_t)\in U$,
 where $U$ is some Zariski dense open subset of $k^t$.

\begin{Lemma}\label{specializ}$[$Specialization Lemma$]$
Let $S$ be  as above. Let  $M$ be a finite $R$-module and $M^{\prime}=M\otimes_R S$.
  Let $N^{\prime} \subseteq M^{\prime}$ be a submodule
such that
$\lambda_{S_{\m S}}(M^{\prime}_{\m S}/N^{\prime}_{\m S})=\delta\in \mathbb N_0 .$ 
 Then
\begin{itemize}
\item[(a)]  For a general vector $\underline{\alpha} \in R^t$,  one has that $\lambda_R(M/N^{\prime}_{\underline{\alpha}})=\delta$.
\item[(b)]   Assume $R$ is equicharacteristic and fix any vector $\underline{\alpha_0}\in R^t$. Then for a general vector $\underline{\alpha}  \in R^t$, one has that $\delta=\lambda_R(M/N^{\prime}_{\underline{\alpha}})\leq \lambda_R(M/N^{\prime}_{\underline{\alpha_0}})$.
\end{itemize}
\end{Lemma}

\demo
We may pass to the $\m$-adic completion of $R$ to assume that $R\cong A/H$, where $A$ is a regular local ring.
We may also replace $R$ by $A$ to assume that $R$ is a regular local ring, and therefore $M^{\prime}$ is a finite module of a polynomial ring over a regular local ring. 
We use induction on $\delta$ to prove part~(a). Notice that this statement holds if $\delta\leq 1$. Indeed, if $\delta=0$, i.e.,
$\lambda_{S_{\m S}}(M^{\prime}_{\m S}/N^{\prime}_{\m S})=0$,
then there exists a polynomial $f\in S\setminus \m S$ such that $fM^{\prime}\subseteq N^{\prime}$. Let $\overline{f}$ be the image of $f$ in
$k[z_1, \ldots, z_t]$ and notice that $\overline{f}\neq 0$. Thus $U=D(\overline{f})$
is a Zariski dense open subset of $k^t$. If $(\overline{\underline{\alpha}})\in U$ then $f(\underline{\alpha})$ is a unit
in $R$. Thus $f(\underline{\alpha})M=(fM^{\prime})_{\underline{\alpha}}\subseteq N^{\prime}_{\underline{\alpha}}$ implies
$M\subseteq N^{\prime}_{\underline{\alpha}}$.

We now consider the case $\delta=1$, i.e., $M^{\prime}_{\m S}/N^{\prime}_{\m S}\cong k(\underline{z})$. In this case there exists $\xi \in M\backslash N^{\prime}$ such that $M^{\prime}_{\m S}=\xi S_{\m S}+N^{\prime}_{\m S}$ and $\xi \m S_{\m S}\subseteq N^{\prime}_{\m S}$. Hence, there exists a polynomial $f\in S\setminus \m S$ such that $
fM^{\prime}\subseteq \xi S+N^{\prime}$ and $f\xi \m S  \subseteq N^{\prime}$. Again $0\neq\overline{f}$ is the image of $f$ in
$k[z_1, \ldots, z_t]$ and $U=D(\overline{f})$ is a Zariski dense open subset of $k^t$. If $(\overline{\underline{\alpha}})$ lies in $U$ then $f(\underline{\alpha})$ is a unit in $R$, and therefore we have $M=M^{\prime}_{\underline{\alpha}}=\xi R+ N^{\prime}_{\underline{\alpha}}$ and $\xi \m \subseteq N^{\prime}_{\underline{\alpha}}$.

 Set $I={\rm Ann}_S(\xi S+ N^{\prime})/N^{\prime}$ and notice that $I_{\m S}=\m S_{\m S}$. By  the properties of general specialization (see \cite{NT}, which still hold because we are over a regular local ring),  there exists a Zariski dense open subset $V$ of $k^t$ such that for every $\overline{{\underline{\alpha}}} \in V$ one has
$$M/N^{\prime}_{\underline{\alpha}}=M^{\prime}_{\underline{\alpha}}/N^{\prime}_{\underline{\alpha}} = \xi R+ N^{\prime}_{\underline{\alpha}}/ N^{\prime}_{\underline{\alpha}}\cong
 R/{\rm Ann}_R[(\xi R+ N^{\prime}_{\underline{\alpha}})/N^{\prime}_{\underline{\alpha}}]
\cong [S/I]_{\underline{\alpha}}.$$
Therefore, we have
$$\begin{array}{lll}
\lambda_R(M/N^{\prime}_{\underline{\alpha}}) 
&= \lambda_R([S/I]_{\underline{\alpha}}) &=\lambda_R(S/I\otimes_S S/(\underline{z}-\underline{\alpha}))\\
& =\lambda_R(S/I\otimes_S S_{\m S}/(\underline{z}-\underline{\alpha})S_{\m S})
& =\lambda_R(S_{\m S}/I_{\m S}\otimes_{S_{\m S}} S_{\m S}/(\underline{z}-\underline{\alpha})S_{\m S})\\
& =\lambda_R(S_{\m S}/\m S_{\m S} \otimes_{S_{\m S}} S_{\m S}/(\underline{z}-\underline{\alpha})S_{\m S})
& =\lambda_R(S/\m S \otimes_{S} S/(\underline{z}-\underline{\alpha})S\otimes_{S}S_{\m S})\\
& =  \lambda_R(k) & = 1.
\end{array}$$

We may then assume $\delta>1$ and assertion (a) holds for $\delta-1$. For every element $x\in M$,  write $x^{\prime}=x\otimes 1$ in $M\otimes_R S=M^{\prime}$. Notice that if every element $x\in M$ has the property that $\frac{x^{\prime}}{1} \in N^{\prime}_{\m S}$, then $M^{\prime}_{\m  S}=N^{\prime}_{\m  S}$, showing that $\delta=0$, which is a contradiction.

Hence, there exists an element $x_0\in M$ with $\frac{x_0^{\prime}}{1}\notin N^{\prime}_{\m  S}$.
We claim that we can choose $x\in M$ with the property that $\frac{x^{\prime}}{1}\in M^{\prime}_{\m  S}\setminus N^{\prime}_{\m S}$ and $x^{\prime}\m  S_{\m  S}\subseteq N^{\prime}_{\m  S}$.
Let $\gamma={\rm max}\{n \in \mathbb N\,|\, x_0^{\prime}\m^{n}S_{\m  S}\not\subseteq N^{\prime}_{\m  S} \}$.
Notice that $0\leq \gamma \leq \delta-1$.
If for every element $a\in \m^{\gamma}R$ we have $\frac{(ax_0)^{\prime}}{1} \in N^{\prime}_{\m  S}$, then $x_0^{\prime}\m^{\gamma}S_{\m  S}\subseteq N^{\prime}_{\m  S}$, that is a contradiction.
Hence, there exists an element $a\in \m ^{\gamma}R$ with $\frac{(ax_0)^{\prime}}{1} \in M^{\prime}_{\m  S} \setminus N^{\prime}_{\m  S}$, but $(ax_0)^{\prime} \m  S_{\m  S} \subseteq N^{\prime}_{\m  S}$. Since $a\in \m^{\gamma}$ and $x_0 \in M$, it follows that $x=ax_0 \in M$ and has the property that $\frac{x^{\prime}}{1}\in M^{\prime}_{\m  S}\setminus N^{\prime}_{\m  S}$ and $x^{\prime}\m  S_{\m  S} \subseteq N^{\prime}_{\m  S}$.
Now, set $N_{\delta-1}^{\prime}=N^{\prime} + x^{\prime}S\subseteq M'$. By construction, we have  $$\lambda_{S_{\m S}}((N_{\delta-1}^{\prime}/N^{\prime})_{\m  S})= 1.$$ Since
$(N_{\delta-1}^{\prime}/N^{\prime})_{\m  S}\cong (x^{\prime}S/(x^{\prime}S \cap N^{\prime}))_{\m  S}$ and $x'S=xR\otimes_R S$ for an $R$-submodule $xR\subseteq M$, by the case of $\delta=1$,  we have
$\lambda_R((N_{\delta-1}^{\prime}/N^{\prime})_{\alpha})= 1$ for a general $\alpha$.
Also, by the above, we obtain
$\lambda_{S_{\m S}}((M^{\prime}/N_{\delta-1}^{\prime})_{\m  S})= \lambda_{S_{\m S}}((M^{\prime}/N^{\prime})_{\m  S})-\lambda_{S_{\m S}}((N_{\delta-1}^{\prime}/N^{\prime})_{\m  S})=\delta -1.$
Hence, by induction hypothesis, for a general vector $\alpha$, one has
$$ \lambda_R((M^{\prime}/N_{\delta-1}^{\prime})_{\alpha})= \delta -1,$$
proving that, for a general vector $\alpha$, we have $$\lambda_R(M/N^{\prime}_{\underline{\alpha}})=\lambda_R((M^{\prime}/N^{\prime})_{\alpha})=\lambda_R((M^{\prime}/N_{\delta-1}^{\prime})_{\alpha})+
\lambda_R((N_{\delta-1}^{\prime}/N^{\prime})_{\alpha})=\delta -1 + 1 = \delta.$$

\bigskip

To prove part (b), first notice that if $\lambda_R(M/N^{\prime}_{\underline{\alpha_0}})=\infty$ then there is nothing to prove. Hence we may assume $\lambda_R(M/N^{\prime}_{\underline{\alpha_0}})<\infty$.

Since $\lambda_{S_{\m S}}(M^{\prime}_{\m S}/N^{\prime}_{\m S})< \infty$ and $\lambda_R(M/N^{\prime}_{\underline{\alpha_0}})<\infty$, there exists a positive integer $t_0$ such that $\m^{t_0}M^{\prime}_{\m S}\subseteq N^{\prime}_{\m S}$ and $\m^{t_0}M\subseteq N^{\prime}_{\underline{\alpha_0}}$. Thus, there exists an element $f\in S \backslash \m S$ such that $f \m^{t_0} M^{\prime}\subseteq N^{\prime}$. Then for every $\underline{\alpha}\in D(\overline{f})$, one has that $\m^{t_0}M\subseteq N^{\prime}_{\underline{\alpha}}$.

Since $R$ is equicharacteristic,   $R$ contains its residue field $k$. Theretofore for every $\underline{\alpha}$ in $U_1=D(\overline{f})\cup \{\underline{\alpha_0}\}$, we have the following isomorphisms of $S/\m^{t_0}S$-modules:
{\small $$\begin{array}{lll}
M^{\prime}/N^{\prime} \otimes_{S}
S/\m^{t_0} S \otimes_{k[\underline{z}]_{(\underline{z}-\underline{\alpha})k[\underline{z}]}} k[\underline{z}]_{(\underline{z}-\underline{\alpha})k[\underline{z}]}/(\underline{z}-\underline{\alpha})
& \cong & M^{\prime}/N^{\prime} \otimes_{S}
S/\m^{t_0} S \otimes_{k[\underline{z}]} k[\underline{z}]/(\underline{z}-\underline{\alpha})\\
& \cong & S/\m^{t_0} S \otimes_{S}
 ( M^{\prime}/N^{\prime}\otimes_{k[\underline{z}]} k[\underline{z}]/(\underline{z}-\underline{\alpha}))\\
& \cong & S/\m^{t_0} S \otimes_{S}
  M/N^{\prime}_{\underline{\alpha}}\\
  &\cong & M/N^{\prime}_{\underline{\alpha}}
  \end{array}
$$}
where the last isomorphism follows because $ M/N^{\prime}_{\underline{\alpha}}$ is an $S/\m^{t_0} S$ module.
Notice that $M^{\prime}/N^{\prime} \otimes_S S/\m^{t_0}S$ is a finite $k[\underline z]$-module ($M^{\prime}/N^{\prime}$ is a finite $S$-module and $S/\m^{t_0}S \cong R/\m^{t_0}\otimes_{k}k[\underline z]$ is a finite $k[z]$-module).
Hence for every $\underline \alpha \in U_1$ we have
{\small$$\begin{array}{lll}
&&\mu_{k[\underline{z}]_{(\underline{z}-\underline{\alpha})k[\underline{z}]}}((M^{\prime}/N^{\prime} \otimes_{S}
S/\m^{t_0} S)_{(\underline{z}-\underline{\alpha})k[\underline{z}]})\\
& = & \lambda_{k[\underline{z}]_{(\underline{z}-\underline{\alpha})k[\underline{z}]}}(M^{\prime}/N^{\prime} \otimes_{S}
S/\m^{t_0}S \,\otimes_{k[\underline{z}]_{(\underline{z}-\underline{\alpha})k[\underline{z}]}} k[\underline{z}]_{(\underline{z}-\underline{\alpha})k[\underline{z}]}/(\underline{z}-\underline{\alpha}))\\
& = & \lambda_R(M/N^{\prime}_{\underline{\alpha}}),\end{array}$$}
where the last equality follows by the above isomorphisms, and the first equality holds by Nakayama's Lemma (that can be applied because $(M^{\prime}/N^{\prime} \otimes_{S}S/\m^{t_0} S )_{(\underline{z}-\underline{\alpha})k[\underline{z}]}$ is finite over $k[z]_{(\underline{z}-\underline{\alpha})k[\underline{z}]}$ by the above).
Set $q=\lambda_R(M/N^{\prime}_{\underline{\alpha_0}})$, then one has  the following Zariski open subset of $k^t$
$$U_2=\{ \underline{\alpha} \in k^t \,| \, \mu_{k[\underline{z}]_{(\underline{z}-\underline{\alpha})k[\underline{z}]}}((M^{\prime}/N^{\prime} \otimes_{S}
S/\m^{t_0} S )_{(\underline{z}-\underline{\alpha})k[\underline{z}]})\leq q\}
= k^t \setminus V({\rm Fitt}_q(M^{\prime}/N^{\prime}\otimes_{S} S/\m^{t_0} S )).
 $$
Notice $U_2$ is dense because $\underline{\alpha_0} \in U_2$.
Finally for any $\underline{\alpha} \in U=U_1\cap U_2$ which is again a  Zariski dense open subset of $k^t$,
 we have
$$\lambda_R(M/N^{\prime}_{\underline{\alpha}} )= \mu_{k[\underline{z}]_{(\underline{z}-\underline{\alpha})k[\underline{z}]}}((M^{\prime}/N^{\prime} \otimes_{S}
S/\m^{t_0} S)_{(\underline{z}-\underline{\alpha})k[\underline{z}]}) \leq  q=\lambda_R(M/N^{\prime}_{\underline{\alpha_0}}).$$

\QED
\bigskip

Lemma \ref{specializ} greatly enhances our ability to study arbitrary ideals and modules. Indeed we are going to apply it to study the index of nilpotancy of any ideal. For this purpose,  we recall that
  the {\it index of nilpotency} of an $R$-ideal $I$ with respect to a reduction $J$ is defined to be the integer
$$
s_J(I)={\rm min}\{n\, |\, I^{n+1}\subseteq J\}.
$$

In Proposition \cite[5.3.3]{F},  Fouli proved that the index of nilpotency of $\m$-primary ideals
 over an equicharacteristic Cohen-Macaulay local ring does
 not depend on the general minimal reduction, and general minimal
 reductions achieve the largest possible index of nilpotency.
 We generalize this result to non $\m$-primary ideals using  Lemma \ref{specializ} as a crucial ingredient
 (see    the following proposition).

\begin{Proposition}\label{nilpotency2}
Assume  $R$ is  Cohen-Macaulay. Let  $I$ be an $R$-ideal which has $\ell(I)=d$ and  the $G_d$ condition.
Let  $J$ be a general minimal reduction of $I$.  Then $s_J(I)$ does not depend on a choice of $J$.
Furthermore, assume $R$ is equicharacteristic, and   either $I$ is $\m$-primary or
 $I$ satisfies  $AN^-_{d-2}$,
${\rm depth}\,(R/I)\geq  1$, and $J\cap I^2=JI$. Let $H$ be any fixed minimal reduction of $I$.  Then $s_H(I)\leq s_J(I)$.
\end{Proposition}

\demo
First by the following exact sequence
$$
 0\rightarrow J+I^{s+1}/J \rightarrow I/J \rightarrow I/J+I^{s+1} \rightarrow 0,
$$
one has that \,  $\lambda(J+I^{s+1}/J)=\lambda(I/J)-\lambda(I/J+I^{s+1})$.  By Lemma \ref{specializ} (see also Proposition \ref{generic1} in Section 5), the lengths   $\lambda(I/J)$ and $\lambda(I/J+I^{s+1})$ do not depend on $J$. Therefore  $\lambda(J+I^{s+1}/J)$ and thus $s_J(I)$ do not depend on a choice of the general minimal reduction $J$.

Next assume $R$ is equicharacteristic. The case where $I$ is $\m$-primary has been proved  by Proposition \cite[5.3.3]{F}.
So we may assume that
$I$ satisfies  $AN^-_{d-2}$,
${\rm depth}\,(R/I)\geq  1$, and $J\cap I^2=JI$.
Write $J=(x_1, \ldots, x_d)$, where $x_1, \ldots, x_d$ are general elements in $I$.
Set $J_{d-1}=(x_1, \ldots, x_{d-1})$.
By \cite[Lemma 3.2]{PX1}, \,  $J_{d-1}: I$ is a geometric $d-1$-residual intersection of $I$,  $(J_{d-1}: I)\cap I=J_{d-1}$, and $(J_{d-1}: I)\cap I^2=J_{d-1}\cap I^2=J_{d-1}I$. Since $I$ satisfies  $AN^-_{d-2}$, one has that $J_{d-1}: I^{\infty}= J_{d-1}: I$.
 Let $\overline{R}=R/J_{d-1}:I$ and use  $^{^{\tratto}}$   to denote images in the quotient ring  $\overline{R}$.
  By the above and the proof of \cite[Proposition 2.1]{PX1}, one has
 $$
 j(I)=e(\overline{I}, \overline{R})
 =\lambda(\overline{I}/x_d \overline{I}) =\lambda(I/(J_{d-1}: I)\cap I +x_d I)
 $$
 $$
=\lambda(I/J_{d-1} +x_d I)=\lambda(I/J_{d-1} +I^2)+
\lambda(J_{d-1} +I^2/J_{d-1} +x_d I) 
 $$
 $$
=\lambda(I/J_{d-1}+I^2)+\lambda(I^2/J_{d-1}\cap I^2+x_dI)=\lambda(I/J_{d-1}+I^2)+\lambda(I^2/JI).
$$

Now consider the minimal reduction $H$ of $I$. Since $H:I=R$, one has that ${\rm ht}\,H:I=\infty$.  By \cite[Lemma 1.4]{U}, one can chose elements  $y_1, \ldots, y_d$ in $H$ such that $H=(y_1, \ldots, y_d)$ and
$H_{d-1}:I$,\, where $H_{d-1}=(y_1, \ldots, y_{d-1})$, \, is a geometric $d-1$-residual intersection of $I$.
Since $I$ satisfies  $AN^-_{d-2}$ and
${\rm depth}\,(R/I)\geq 1$, \,one has that $H_{d-1}: I^{\infty}= H_{d-1}: I$,
  $(H_{d-1}: I)\cap I=H_{d-1}$, and $(H_{d-1}: I)\cap I^2=H_{d-1}\cap I^2=H_{d-1}I$ (see \cite[Lemmas 2.3 and 2.4]{JU}).
  Moreover by avoiding finitely many more prime ideals, one can also assume that $y_1, \ldots, y_d$ form a super-reduction for $I$ (in the sense of Achilles and Manaresi \cite{AM}). Therefore we can use $y_1, \ldots, y_d$ to compute the $j$-multiplicity $j(I)$ \cite[3.8]{AM}.
 Let $R^{\prime}=R/H_{d-1}:I$ and use  $^{\prime}$   to denote images in the quotient ring  $R^{\prime}$.
  By the same argument as above,  one has
 $$
 j(I)=e(I^{\prime}, R^{\prime})
 =\lambda(I^{\prime}/x_d I^{\prime})=\lambda(I/H_{d-1}+I^2)+\lambda(I^2/HI).
 $$

By Lemma \ref{specializ} (see also Proposition \ref{generic1} in Section 5), one has that $\lambda(I/J_{d-1}+I^2)\leq
\lambda(I/H_{d-1}+I^2)$ and $\lambda(I^2/JI)\leq \lambda(I^2/HI)$, thus, $\lambda(I^2/JI)=\lambda(I^2/HI)$.
From the exact sequences
\begin{eqnarray*}
& 0\rightarrow HI+I^{s+1}/HI \rightarrow I^2/HI \rightarrow I^2/HI+I^{s+1} \rightarrow 0,\\
& 0\rightarrow JI+I^{s+1}/JI \rightarrow I^2/JI \rightarrow I^2/JI+I^{s+1} \rightarrow 0,
\end{eqnarray*}
one  deduces
$\lambda(HI+I^{s+1}/HI)\leq\lambda(JI+I^{s+1}/JI)$.
Let $s=s_J(I)$. If $s=0$, the statement follows. Otherwise, one has $I^{s+1}\subseteq J\cap I^2=JI$, whence $\lambda(JI+I^{s+1}/JI)=0$. Therefore, one obtains the equality $\lambda(HI+I^{s+1}/HI)=0$, which, in turn, implies $I^{s+1}\subseteq HI\subseteq H$.
\QED
\bigskip

Let  $I$ be an $R$-ideal which has $\ell(I)=d$ and  the $G_d$ condition. We then define the 
{\it index of nilpotency} of $I$ as 
$
s(I)=s_J(I),
$ 
where   $J$ is a general minimal reduction of $I$.  This number is well-defined by Proposition \ref{nilpotency2}.  
We set two typical settings  for our next results.
\begin{Setting}\label{ass}
Let   $R$ be  Cohen-Macaulay and    $I$   an $R$-ideal.  Assume  either
\begin{enumerate}
\item    $\ell (I)=d$ and $I$ satisfies  $G_{d}$ condition,
  $AN^-_{d-2}$, and
  ${\rm depth}\,(R/I)\geq {\rm min}\{{\rm dim}\,R/I, 1\}$. 
\item  or $I$ is $\m$-primary and $J_{d-1}\cap I^2\subseteq JI$, where $J=(x_1, \ldots, x_d)$ is a general minimal reduction of $I$ and $J_{d-1}=(x_1, \ldots, x_{d-1})$.
\end{enumerate}
\end{Setting}

The following lemmas generalize to $j$-stretched ideals  the corresponding results of stretched $\m$-primary ideals proved in \cite{RV3}. Since the associated graded rings of ideals having minimal $j$-multiplicity are known to be Cohen-Macaulay \cite[Theorem~3.9]{PX1}, we can harmlessly assume that $I$ does not have minimal $j$-multiplicity.
\begin{Lemma}\label{Properties}
Let  $R$ and $I$ be as in Setting \ref{ass}. 
Let $I$ be $j$-stretched, not having minimal $j$-multiplicity.
 Then
 \begin{itemize}
\item[(a)] $j(I)\geq \lambda(\overline{R}/\overline{I})+h+1$, where $h=\lambda(\overline{I}/\overline{I^2})-\lambda(\overline{R}/\overline{I}),\, \overline{R}=R/J_{d-1}: I^{\infty}.$
\item[(b)] For every $n\geq 1$, we have $I^{n+1}=JI^n+(a^nb)$, where $a,b\in I$ and $a,b \notin J$.
\item[(c)] For every $n\geq 1$, we have $a^nb\,\m \subseteq I^{n+2}+JI^n$.
\item[(d)] $I=(b)+(J: a)\cap I$.
\end{itemize}
\end{Lemma}

\demo
(a)  By \cite[Proposition~2.1]{PX1}, we have $j(I)=e(I,\overline{R})= \lambda(\overline{I}/\overline{I^2})+\lambda(\overline{I^2}/x_d\overline{I}).$
By definition of $h$, this equals $\lambda(\overline{R}/\overline{I})+h+\lambda(\overline{I^2}/x_d\overline{I})$.
Hence, to finish the proof of (a), we have to show that $\lambda(\overline{I^2}/x_d\overline{I})\geq 1$, which
holds because the length is $0$ if and only if $I$ has minimal $j$-multiplicity.

To prove assertion (b), we first  show  $\lambda(I^2/JI+I^3)=1$.
Since  $I$ is $j$-stretched and does not have minimal $j$-multiplicity, one has that
$\lambda (\overline{I}^2/x_d \overline{I}+ \overline{I}^3)=1$ (otherwise $\overline{I}^2=x_d \overline{I}+ \overline{I}^3$ which implies $\overline{I}^2=x_d \overline{I}$ by Nakayama's Lemma).
Notice $\lambda (\overline{I}^2/x_d \overline{I}+ \overline{I}^3)=
\lambda [I^2/((J_{d-1}:I^{\infty})\cap I^2 + x_d I +I^3)]=1$.  We  need to show  $(J_{d-1}: I^{\infty}) \cap I^2 = J_{d-1}I$, which immediately yields $ \lambda(I^2/JI + I^3)= 1$. The case where $I$ is  not $\m$-primary has been proved by \cite[Lemma 3.2]{PX1}. So assume $I$ is $\m$-primary. Since $J_{d-1}\cap I^2 \subseteq JI$ and $x_d$ is a non zero divisor on $R/J_{d-1}$, one has
$$
(J_{d-1}:I)\cap I^2=J_{d-1}\cap I^2=J_{d-1} \cap JI=J_{d-1} \cap (J_{d-1}I + x_dI)  = J_{d-1}I + J_{d-1} \cap x_dI
$$
$$
                                    = J_{d-1}I + x_d(J_{d-1} :_{I} x_d)  = J_{d-1}I  + x_dJ_{d-1}
                                   = J_{d-1}I.
$$
Now we use  induction on $n$ to prove assertion (b). The length $\lambda(I^2/JI+I^3)=1$ implies that $I^2=JI+I^3+(ab)$ for some
$a,b\in I \setminus J$. By Nakayama's Lemma, $I^2=JI+(ab)$, proving the statement in the case $n=1$. For any $n\geq 1$,
assume  $I^{n+1}=JI^n+(a^nb)$. We need to show that $I^{n+2}=JI^{n+1}+(a^{n+1}b).$
This holds since $I^{n+2}=I(JI^n+(a^nb))=JI^{n+1}+a^n(bI) \subseteq JI^{n+1}+a^n(JI+(ab))=JI^{n+1}+(a^{n+1}b)$.

The proofs of (c) and (d) are similar to the corresponding statements in  \cite[Lemma 2.4]{RV3}. We write them for the sake of completeness.
Assertion (c) can be proved by induction on $n\geq 1$. The case $n=1$ follows from the facts that $\lambda(I^2/JI+I^3)=1$ and $I^2=JI+(ab)$.
Now assume $n\geq 1$ and $(a^nb)\m\subseteq I^{n+2}+JI^n$. Then one has
$$(a^{n+1}b)\m=a[(a^nb)\m]\subseteq a[I^{n+2}+JI^n]\subseteq I^{n+3}+JI^{n+1}.$$

(d) Since $aI\subseteq I^2=JI + (ab)$, we have $I\subseteq (JI+(ab)): a$. The easily checked equality $(JI+(ab)):a=(JI: a) + (b)$ now implies $$I\subseteq [(JI: a) + (b)]\cap I\subseteq (b)+(J: a)\cap I\subseteq I.$$
\QED

\bigskip

Let $I$ be an $R$-ideal which has  maximal analytic spread $\ell (I)=d$ and  the $G_d$ condition. Let  
$J=(x_1, \ldots, x_d)$ be a general minimal reduction of $I$.  
Recall  $\overline{R}=R/J_{d-1}:I^{\infty}$, where $J_{d-1}=(x_1, \ldots, x_{d-1})$. 
 We set $\nu_n=\lambda(I^{n+1}/JI^n)$ and $\overline{\nu_n}=\lambda(\overline{I^{n+1}}/\overline{JI^n})$ for every $n\geq 0$, which are  well-defined by Lemma \ref{specializ}.

\begin{Lemma}\label{non-inc}
Let  $R$ and $I$ be as in Setting \ref{ass}.
If $I$ is $j$-stretched, then
\begin{itemize}
\item[(a)] $\nu_n\leq \nu_{n-1}$ \, for every $n\geq 2$.
\item[(b)] $\overline{\nu_1}=\nu_1$ and $\overline{\nu_n}\leq \nu_n$\, for every $n\geq 2$.
\end{itemize}
\end{Lemma}

\demo
(a) If $I$ has minimal $j$-multiplicity then $r(I)\leq 1$ (see \cite[Theorem 3.3]{PX1}). Hence $\nu_n=\lambda(I^{n+1}/JI^n)=0$ for every $n\geq 1$. Assume $I$ does not have minimal $j$-multiplicity.
Let $a$ be the same  as in Lemma \ref{Properties}~(b). 
Then for every $n\geq 2$, we have the following epimorphism
$$I^{n}/JI^{n-1} \stackrel{\cdot a}{\longrightarrow} I^{n+1}/JI^{n}\longrightarrow 0.$$
Hence $\nu_n=\lambda(I^{n+1}/JI^n)\leq \nu_{n-1}=\lambda(I^{n}/JI^{n-1})$.

(b) For any $n$, we have the natural epimorphism
$$I^{n+1}/JI^n\longrightarrow \overline{I^{n+1}}/\overline{JI^n} \longrightarrow 0,$$
inducing the inequality $\overline{\nu_n}\leq \nu_n$. Furthermore, one has
$$\begin{array}{lll}
\overline{\nu_1}&=\lambda(\overline{I^2}/\overline{JI}) &=\lambda(I^2/(J_{d-1}: I^{\infty})\cap I^2 + JI)\\
                &=\lambda(I^2/JI)                       &=\nu_1,
\end{array}$$
where the third equality follows from the fact $(J_{d-1}: I^{\infty})\cap I^2= J_{d-1}I$ (see \cite[Lemma~3.2]{PX1}
and the proof of Lemma \ref{Properties}).

\QED
\bigskip

 Now consider the Hilbert function $H_{I,\,\overline{R}}(n)=\lambda(\overline{I^n}/\overline{I^{n+1}})$, which
 does not depend on $J$ (see \cite{PX2}). In particular, it is well-defined  the integer
 $h=\lambda(\overline{I}/\overline{I^2})-\lambda(\overline{R}/\overline{I})$, which is dubbed the {\it embedding codimension} of $I$. Moreover, one has that (see \cite{PX1} and \cite{RV3})
$$j(I)=e(I, \overline{R})=\lambda(\overline{R}/\overline{I})+h+K-1,
\mbox{ where } K-1=\lambda(\overline{I^2}/x_d\overline{I}).$$

The following corollary shows that if $I$ is $j$-stretched, then $K$ is the index of nilpotency $s(I)$.
\begin{Corollary}\label{K}
Let  $R$ and $I$ be as in Setting \ref{ass}.
If $I$ is $j$-stretched,
then
$$
\quad\quad  \nu_1=K-1,\quad\quad I^{K}\nsubseteq J,\quad\quad I^{K+1}\subseteq J.
$$
\end{Corollary}

\demo
By the proof of  Lemma \ref{non-inc} (b),  we have $K-1=\lambda(\overline{I^2}/x_d\overline{I})=\lambda(I^2/JI)=\nu_1$.
Since $I$ is $j$-stretched, one has that
$$ P_{I,\,\overline{R}/x_d\overline{R}}=\lambda(\overline{R}/\overline{I})+hz+z^2+
 \cdots+z^{j(I)-h+1-\lambda(\overline{R}/\overline{I})}.$$
Therefore $K$ is the least positive integer with
$$I^{K+1}\subseteq [(J_{d-1}: I^{\infty})\cap I^{K+1}]+I^{K+2}+J\subseteq I^{K+2}+J.$$ By Nakayama's Lemma, $K$ is the least positive integer with $I^{K+1}\subseteq J$.
\QED
\bigskip

The next result is the last ingredient that we need to characterize the Cohen-Macaulayness of ${\rm gr}_{I}(R)$ when $I$ is $j$-stretched. It shows that the inclusion  $I^{K+1}\subseteq JI^n$ is equivalent to certain Valabrega-Valla equalities for small powers of $I$. More precisely,
\begin{Proposition}\label{VV}
Let  $R$ and $I$ be as in Setting \ref{ass}. If $I$ is $j$-stretched with  index of nilpotency $K$, then for  any $0\leq n\leq K$, one has:
 \begin{itemize}
\item[(a)] $J\cap I^{n+1}=JI^n+(a^Kb)$, where $a$ and $b$ are as in Lemma \ref{Properties} (b).
\item[(b)]  $I^{K+1}\subseteq JI^n$ if and only if $J\cap I^{t+1}=JI^t$ for every $t\leq n$.
\end{itemize}
\end{Proposition}

\demo
(a) We use descending induction
on $n\leq K$. When $n=K$, by Lemma  \ref{Properties} (b),  $J\cap I^{K+1}=I^{K+1}=JI^K + (a^K b)$.
Now assume $J\cap I^{n+1}=JI^n+(a^Kb)$ and prove $J\cap I^{n}=JI^{n-1}+(a^Kb)$. One inclusion
$JI^{n-1}+(a^Kb)\subseteq J\cap I^{n}$ is clear. We prove $J\cap I^{n}\subseteq JI^{n-1}+(a^Kb)$.  By Lemma \ref{Properties} (b), $J\cap I^{n}=J\cap(JI^{n-1}+(a^{n-1}b))=JI^{n-1} + (a^{n-1} b) \cap J$.
 Since  $I^{n}\not\subseteq J$ by Corollary \ref{K}, one has
$$(a^{n-1}b) \cap J\subseteq a^{n-1}b\,\m \cap J\subseteq(I^{n+1} + JI^{n-1}) \cap J
$$
$$=(I^{n+1} \cap J ) + JI^{n-1}
=JI^{n}+(a^Kb)+ JI^{n-1}=JI^{n-1}+(a^Kb).$$

The proof of assertion (b) is similar to the one of \cite[Lemma~2.5 (ii)]{RV3}.
\QED
\bigskip

\section{Cohen-Macaulayness and almost Cohen-Macaulayness of ${\rm gr}_I(R)$}

In this section we study the depth of the associated graded rings ${\rm gr}_I(R)$ of $j$-stretched ideals. In Theorem \ref{CM}, we prove that ${\rm gr}_I(R)$ is Cohen-Macaulay if  and only if  the reduction number and the index of nilpotency of the ideal $I$ are equal. We also prove Sally's conjecture for $j$-stretched ideals, providing a  sufficient condition for ${\rm gr}_I(R)$  to be almost Cohen-Macaulay (see Theorem \ref{2}).  Our work combines the approaches of Rossi-Valla and Polini-Xie and generalizes
widely the main results of \cite{S3}, \cite{RV3} and \cite{PX1}.
\medskip

\begin{Theorem}\label{CM}
Let  $R$ and $I$ be as in Setting \ref{ass}. Let  $I$ be $j$-stretched with  the index of nilpotency $K$.
 Then the following statements are equivalent:
 \begin{itemize}
\item[(a)]   ${\rm gr}_I(R)$ is Cohen-Macaulay.
\item[(b)]   $r(I)=K.$
\end{itemize}

Furthermore, if $R$ is equicharacteristic, then $($a$)$ and $($b$)$ are also equivalent to
\begin{itemize}
\item[(c)]   $I^{K+1}=HI^K$ for some minimal reduction $H$ of $I$.
\end{itemize}
\end{Theorem}

\demo
We first prove the following two claims.
\begin{Claim}\label{C1}
 The equalities $J\cap I^{n+1}=JI^n$ hold for every $n\geq 0$ if and only if $I^{K+1}=JI^K$.
\end{Claim}
The forward direction is straightforward since $I^{K+1}\subseteq J$. Conversely, if $I^{K+1}=JI^K$, then by Proposition \ref{VV} (b),  one has $J\cap I^{n+1}=JI^n$ for every $0\leq n\leq K$. If $n\geq K$,  one has $I^{n+1}=I^{n-K}I^{K+1}=I^{n-K}JI^K=JI^n$ and then obtains $J\cap I^{n+1}=JI^n$.

\begin{Claim}\label{C2}
Write  $g= {\rm grade}\,I$. Let $x_1^{*}, \ldots, x_d^{*}$ be the initial forms of $x_1,\ldots,x_d$ in ${\rm gr}_I(R)$.
If $I^{K+1}=JI^K$, then $x_1^{*}, \ldots, x_g^{*}$ \, form a regular sequence on ${\rm gr}_I(R)$.
\end{Claim}

Since $x_1,\dots,x_g$ are general elements in $I$ and $g={\rm grade}\,I$, then $x_1,\dots,x_g$ form a regular sequence on $R$. By Valabrega-Valla criterion (see \cite[Proposition~2.6]{VV} or \cite[Theorem~1.1]{RV}),  we only need to  show $(x_1,\ldots,x_g)\cap I^n=(x_1,\ldots,x_g) I^{n-1}$ for every $n\geq 1$.
The case where $I$ is $\m$-primary  follows from \cite[Theorem 2.6]{RV3}, hence we may assume ${\rm dim}\,R/I>0$.
We use induction on $n$ to prove $(x_1,\ldots,x_i)\cap I^n=(x_1,\ldots,x_i) I^{n-1}$ for every $n\geq 1$ and $0\leq i\leq d$. This is clear if $n=1$. We then assume $n\geq 2$ and the equality holds for $n-1$. Now, we use descending induction on $i\leq d$. Since $I$ is $j$-stretched with $I^{K+1}=JI^K$, then, by Claim \ref{C1}, $J\cap I^{n}=JI^{n-1}$, which proves the case $i=d$. Now assume $i<d$ and, by induction, that $(x_1,\ldots,x_{i+1})\cap I^{n}=(x_1,\ldots,x_{i+1})I^{n-1}$. Then
\vskip 0.1in

\noindent \begin{tabular}{llll}
&&$(x_1,\ldots,x_i)I^{n-1}\subseteq (x_1, \ldots, x_{i})\cap I^{n}$ & \\
&=&$(x_1, \ldots, x_i)\cap (x_1,\ldots,x_{i+1})I^{n-1}$ & 
\mbox{ by induction on }i\\
&=&$(x_1,\ldots,x_i)\cap ((x_1,\ldots,x_{i})I^{j-1}+x_{i+1}I^{n-1})$ &\\
&=& $(x_1,\ldots,x_{i})I^{n-1}+(x_1,\ldots,x_i)\cap x_{i+1}I^{n-1}$&\\
&=& $(x_1,\ldots,x_{i})I^{n-1}+x_{i+1}[((x_1,\ldots,x_{i}):x_{i+1})\cap I^{n-1}]$&\\
&=&$(x_1,\ldots,x_{i})I^{n-1}+x_{i+1}[(x_1,\ldots,x_{i})\cap I^{n-1}]$ & 
 \mbox{ by \cite[Lemma~3.2]{PX1}}\\
&= &$(x_1,\ldots,x_{i})I^{n-1}+x_{i+1}(x_1,\ldots,x_{i}) I^{n-2}$ 
& \mbox{ by induction on }n\\
&$\subseteq$ &$(x_1,\ldots,x_i)I^{n-1}$&
\end{tabular}
\vskip 0.1in

\noindent
which yields the desired equality.
\vskip 0.1in

We are now ready to prove the theorem.

(a) $\Longleftrightarrow$ (b). The proof is similar to \cite[Theorem~3.8]{PX1}. Set  $\delta(I)=d-g$. We prove the equivalence of (a) and (b) by induction on $\delta(I)$.  If  $\delta(I)=0$, the assertion follows because we proved in Claim \ref{C2} that $x_1^{*}, \ldots, x_g^{*}$ form a regular sequence on ${\rm gr}_I(R)$.  Thus we may assume  that $\delta(I)\geq 1$ and the theorem holds for smaller values of $\delta(I)$. In particular, $d\geq g+1$. Since in both cases  $x_1^{*}, \ldots, x_g^{*}$ form a  regular sequence on ${\rm gr}_I(R)$, we may factor out $x_1, \ldots, x_g$ to assume $g=0$. Now $d=\delta(I) \geq 1$.
Set    $R^{\prime}=R/H_0$, where $H_0=0: I$, and use $^{\prime}$ to denote images in $R^{\prime}$. 
By \cite[Lemma~3.2]{PX1}, one has $I\cap H_0=0$,  $R^{\prime}$ is Cohen-Macaulay with  ${\rm dim}\,R^{\prime}=d$,
\, ${\rm grade}\,(I^{\prime})\geq 1$,  $\ell(I^{\prime})=d$, $I^{\prime}$ still satisfies  
$G_{d}$ and $AN^-_{d-2}$ on $R^{\prime}$ and ${\rm depth}\,(R^{\prime}/I^{\prime})\geq 
{\rm min}\{{\rm dim} \,R^{\prime}/I^{\prime}, 1\}$.
By the definition of $j$-stretchedness, one has  that $I^{\prime}$ is $j$-stretched in $R^{\prime}$ with $K=s(I^{\prime})$.    
Since $\delta(I^{\prime})=d-{\rm grade}\,(I^{\prime})<d=\delta(I)$, by induction hypothesis, 
 ${\rm depth}({\rm gr}_{I^{\prime}}(R^{\prime}))\geq d$ if and only 
 if ${I^{\prime}}^{K+1}=J^{\prime}{I^{\prime}}^K$.
Because  $I\cap H_0=0$, one has   ${I^{\prime}}^{K+1}/J^{\prime}{I^{\prime}}^K\cong I^{K+1}/JI^K$ and the following exact sequence
\begin{equation}\label{eq11}
0\rightarrow H_0\rightarrow {\rm gr}_I(R)\rightarrow {\rm gr}_{\overline{I}}(\overline{R})\rightarrow 0.
\end{equation}
Notice that ${\rm depth}_{\m G}(R/H_0)= d$. Hence,  we have 
${\rm gr}_I(R)$  is Cohen-Macaulay  if and only if ${\rm gr}_{I^{\prime}}(R^{\prime})$ is 
Cohen-Macaulay if and only if $I^{K+1}=JI^K$, i.e., $r(I)=K$. 

Finally, we assume $R$ is equicharacteristic and  prove  (b) $\Longleftrightarrow$ (c).  Clearly (b) implies (c).  To prove the converse,  notice that, for a general minimal reduction $J$ and a fixed minimal reduction $H$ of $I$, Lemma \ref{specializ} implies that
$\lambda(I^{t+1}/JI^t)\leq \lambda(I^{t+1}/HI^t)$ for $t\geq 0$ (see also the proof of Proposition \ref{generic1} in Section 5). Therefore,
$$
K=s(I)\leq r(I)\leq r_H(I).
$$
If (c) holds then one has $r_H(I)=K$ which, in turn,  yields $K=s(I)= r(I)$.
\QED

\bigskip

As an immediate application, we recover one of the two main results of Polini-Xie.
\begin{Corollary}$($\cite[Theorem~3.9]{PX1}$)$\label{PX1}
Let $R$ be a $d$-dimensional Cohen-Macaulay local ring and $I$  an $R$-ideal with $\ell(I) = d$. Assume ${\rm depth}\, (R/I) \geq {\rm min}\,\{{\rm dim}(R/I),1\}$ and $I$ satisfies $G_d$ and $AN_{
d-2}^-$. If $I$ has minimal $j$-multiplicity then ${\rm gr}_I(R)$ is Cohen-Macaulay.
\end{Corollary}

\demo If $I$ has minimal $j$-multiplicity then $r(I)\leq 1$ (see \cite[Theorem 3.4]{PX1}). Hence  $K=1$ and a straightforward application of Theorem \ref{CM} concludes the proof.
\QED
\bigskip

In the following, we provide   examples of $j$-stretched ideals which satisfy the assumptions of Theorem \ref{CM}, and therefore their  associated graded rings ${\rm gr}_I(R)$ are  Cohen-Macaulay by our theorem. Notice that the reduction number of the $j$-stretched ideal $I$ in Example \ref{ex1} could be arbitrarily large, hence, none of the previous criteria in the literature proves the Cohen-Macaulayness of ${\rm gr}_I(R)$.

\begin{Example}\label{ex1}
Fix any $r\geq 1$. Let $R=\mathbb{C}\llbracket{x,y,z}\rrbracket/(x^{r+1},xz,yz)$ and $I=(x,y)$. We have seen in
  Example \ref{r} that $R$ is
a $1$-dimensional Cohen-Macaulay local ring and $I$ is a Cohen-Macaulay  ideal of height 0 which has $\ell (I)=1$, $G_1$ condition, and  $AN^-_{-1}$. The ideal $I$ is also  $j$-stretched with reduction number $r$ (if $r> 2$ then $I$ does not have almost minimal $j$-multiplicity).
By computations, $s(I)=r=r(I)$. Hence by Theorem \ref{CM}, one has that  ${\rm gr}_I(R)$ is Cohen-Macaulay (indeed, by computations,  ${\rm gr}_I(R)\cong \mathbb{C}[x,y,z,t,u]/(x,y,zu,t^{r+1},zt)$).
\end{Example}

\begin{Example}\label{ex2}
Let $I$ be one of the following ideals:
\begin{itemize}
\item $I\subseteq R=k[a,b,c]$ is the defining ideal of $n=6$ generic points of $\mathbb P^2$.
\item $I\subseteq R=k[a,b,c,d]$ is the defining ideal of $n=4$ or $n=5$ generic points of $\mathbb P^3$.
\item $I=(a^2,ac,bc,bd,cd)\subseteq R=k[a,b,c,d]$.
\item $I=(ab,ac,ad,bc,bd,cd)\subseteq R=k[a,b,c,d]$.
\item $I=(a^2,b^2, ad,bd,cd)\subseteq R=k[a,b,c,d]$.
\item $I=(a^2,b^2,c^2,ab,bc,cd,de)\subseteq R=k[a,b,c,d,e]$.
\end{itemize}
Then, 
 $I$ satisfies all the assumptions of Theorem \ref{CM} and 
$r(I)=2=s(I)$. 
Therefore,  ${\rm gr}_I(R)$ is Cohen-Macaulay.
\end{Example}

The following theorems (Theorems \ref{dim2} and \ref{2})
 provide a sufficient condition for  ${\rm gr}_I(R)$ to be  almost Cohen-Macaulayness, where $I$ is a $j$-stretched ideal. They generalize \cite[Theorem 4.4]{RV}, \cite[Theorem 4.7]{PX1} and \cite[Theorem 4.10]{PX1}.

\begin{Theorem}\label{dim2}
Let $R$ be a $2$-dimensional Cohen-Macaulay local ring with infinite residue field. Let $I$ be a $j$-stretched ideal such that   $\ell (I)=2$,   $I$ satisfies  $G_{2}$ condition and
  $AN^-_{0}$, and ${\rm depth}\,(R/I)\geq {\rm Min}\{{\rm dim}\,R/I, 1\}$.
  Let $J=(x_1, x_2)$ be a general minimal reduction of $I$ and assume there exists a positive integer $p$ such that
  \begin{itemize}
\item[(i)] $\lambda(J\cap I^{n+1}/JI^n)=0$ for every $0\leq n\leq p-1.$
\item[(ii)]  $\lambda(I^{p+1}/JI^p)\leq 1$.
\end{itemize}
Then
\begin{itemize}
\item[(a)] $x_1^*$ is regular on ${\rm gr}_I(R)_+.$
\item[(b)] ${\rm depth}\,( {\rm gr}_I(R) )\geq 1$ .
\end{itemize}
\end{Theorem}

\demo We first prove part (a). 
 If $I$ is $\m$-primary then both claims follow from \cite[Theorem 4.4]{RV}. Thus we may assume that ${\rm dim}\,(R/I)>0$.
Since $\lambda(I^{p+1}/JI^p)\leq 1$,
one has $I^{p+1}=(ab)+JI^p$ for some $a\in I, b\in I^p$ with $ab\notin JI^p$. For $n\geq p$, the multiplication by $a$ gives a surjective map from $I^{n+1}/JI^{n}$ to $I^{n+2}/JI^{n+1}$. Thus the length $\lambda (I^{n+1}/JI^{n})\leq 1$ for every $n\geq p$.

Notice that $x_1$ is regular on $I$, since $(0: x_1)  \cap I =0$ (by \cite[Lemma~3.2]{PX1}).  To prove that  $x_1^*$ is regular on ${\rm gr}_I(R)_+={\rm gr}_I(I)$, we only need to show $x_1 I\cap I^nI=x_1 I^{n-1}I$
for every $n\geq 1$ by \cite[Proposition~2.6]{VV} (see also \cite[Lemma~1.1]{RV}).
This is clear if $n=1$; hence we may  assume $n\geq 2$. Let $^{\prime}$ denote images
in $R^{\prime}=R/(x_1)$ and set  $s=r (I^{\prime})$. We claim that it is enough to show   $r(I)=s$. Indeed, if $r(I)=s$, then $x_1 I\cap I^nI=x_1I\cap JI^{n-1}I$ for every $n\geq 1$. This is clear  if $s\leq p$.
Assume $s>p$. If $n\leq p-1$, then $x_1I\cap I^{n}I=x_1I\cap J\cap I^{n}I=x_1I\cap JI^{n-1}I$.
If $p\leq n\leq s-1$, then 
\vskip -.4cm
\begin{eqnarray*}
0&<&\lambda (I^n I/ JI^{n-1}I+(x_1) \cap I^{n} I)\\
&=&\lambda(I^n I/ JI^{n-1}I) - \lambda(JI^{n-1}I+(x_1) \cap I^{n}I / JI^{n-1}I)\\
&=& 1 - \lambda(JI^{n-1}I+(x_1) \cap I^{n}I/JI^{n-1}I),
\end{eqnarray*}
which yields $JI^{n-1}I+(x_1)\cap I^{n}I=JI^{n-1}I$. 
Furthermore, if $n\geq s= r(I)$, then  $I^nI=JI^{n-1}I$ and, therefore, $(x_1) I\cap I^nI=x_1I\cap JI^{n-1}I$ for every  $n\geq s$. 
Now an argument similar to the one of  \cite[4.7]{PX1} gives  $x_1 I\cap I^nI=x_1 I^{n-1}I$
for every $n\geq 1$.

To complete the proof of part (a), we still need to to show that $r(I)=s$. This follows by an argument similar to the one employed in \cite[4.7]{PX1}. We write it for the sake of completeness. We use a  result on  the Ratliff-Rush filtration
$\widetilde{I^n} I:=\cup_{t\geq 1}(I^{n+t}I:_{I} I^t)$  (see \cite[Theorem~4.2]{RV} or \cite[Corollary 4.5]{PX1}). Since  $x_1$ is regular on $I$,  by  \cite[Lemma~3.1]{RV},
there exists an integer $n_0$ such that
$I^nI=\widetilde{I^n} I$ for $n\geq n_0$, and
\begin{equation}\label{RRF}
\widetilde{I^{n+1}} I  :_{I} x_1=\widetilde{I^{n}}I \quad \mbox{for every } n\geq 0.
\end{equation}

On the quotient ring $R^{\prime}=R/(x_1)$, there are two filtrations:
$$
\mathbb{M}: I^{\prime}\supseteq  {I^{\prime}}^2\supseteq\ldots\supseteq {I^{\prime}}^{n} \supseteq \ldots
$$
and
$$
\mathbb{N}:I^{\prime}\supseteq \widetilde{I} I^{\prime}\supseteq\ldots\supseteq \widetilde{I^{n-1}} I^{\prime}\supseteq \ldots
$$
Notice that $\mathbb{M}$ is an  $I^{\prime}$-adic filtration and $\mathbb{N}$ is a good $I^{\prime}$-filtration (see \cite[page ~9]{RV} for the definition of good filtrations).
Notice   $\lambda (I^{\prime}/ {I^{\prime}}^2) < \infty $.
Since $ I^{n} I^{\prime}=\widetilde{I^{n}}I^{\prime}$ for $n\geq n_0$, the associated graded modules  ${\rm gr}_{\mathbb{M}}(I^{\prime})$ and ${\rm gr}_{\mathbb{N}}(I^{\prime})$
have the same Hilbert coefficients $e_0$ and $e_1$. Since    $I$ contains a non zero divisor 
 on $I^{\prime}$,  by \cite[Lemmas~2.1 and 2.2]{RV}, we have
\begin{eqnarray*} \label{eq3}
&\sum_{n\geq 0}^{p-2} \lambda(I^{n+1}I/JI^nI)+(s-1)-(p-2)\\
&=\sum_{n\geq 0} \lambda({I^{\prime}}^{n+2}/ J^{\prime}{I^{\prime}}^{n+1}) =e_1(\mathbb {M})=e_1(\mathbb {N}) =
\sum_{n\geq 0} \lambda(\widetilde{I^{n+1}}I^{\prime}/ J\widetilde{I^{n}}I^{\prime}).
\end{eqnarray*}

 The first equality follows from the fact that, for $0\leq n \leq s-1$, one has
$$
\lambda(I^{n+1}/JI^{n})=\lambda({I^{\prime}}^{n+1}/J^{\prime}{I^{\prime}}^{n}).
$$
Indeed, if $0\leq n\leq p-1$ then $\lambda(J\cap I^{n+1}/JI^{n})=0$. Therefore,
\begin{eqnarray*}
&\lambda(I^{n+1}/JI^{n})=\lambda(I^{n+1}/J\cap I^{n+1})=\lambda({I^{\prime}}^{n+1}/J^{\prime}\cap {I^{\prime}}^{n+1})\\
&\leq \lambda({I^{\prime}}^{n+1}/J^{\prime}{I^{\prime}}^n)\leq \lambda(I^{n+1}/JI^n).
\end{eqnarray*}
On the other hand, if $p\leq n\leq s-1$, we have $0< \lambda({I^{\prime}}^{n+1}/J^{\prime}{I^{\prime}}^n) \leq \lambda(I^{n+1}/JI^n)=1$. This proves that $\lambda({I^{\prime}}^{n+1}/J^{\prime}{I^{\prime}}^n) = \lambda(I^{n+1}/JI^n)=1$ for $p\leq n\leq s-1$
 and $\lambda({I^{\prime}}^{n+1}/J^{\prime} {I^{\prime}}^n) =0$ for $n \geq s$.
 \medskip

We now prove that $ \lambda(\widetilde{I^{n+1}}I^{\prime}/ J\widetilde{I^{n}}I^{\prime}) =  \lambda(\widetilde{I^{n+1}} I/ J \widetilde{I^n}I)$ for every $n\geq 0$.
Since
$$\widetilde{I^{n+1}}I^{\prime}/ J\widetilde{I^{n}} I^{\prime}\cong \widetilde{I^{n+1}} I/((x_1)\cap \widetilde{I^{n+1}}I+ x_2\widetilde{I^{n}}I),
$$
we just need to show $(x_1)\cap \widetilde{I^{n+1}}I=(x_1)\widetilde{I^{n}} I$.
We first prove $(x_1)\cap \widetilde{I}I=x_1I$. Since $(x_1)\cap \widetilde{I}I\supseteq x_1I$, it suffices to show the equality locally at every associated prime ideal of $R/x_1I$. By Lemma \cite[3.2]{PX1}, every $\p\in {\rm Ass}(R/x_1I)$ is not maximal. Hence $(x_1)_\p=\widetilde{I}_\p=I_\p$ and $(x_1)_\p\cap \widetilde{I} I_\p= \widetilde{I} I_\p=x_1I_\p$. This shows $(x_1)\cap \widetilde{I}I=x_1I$. Now for any $n\geq 1$, $(x_1)\cap \widetilde{I^{n+1}}I=x_1I \cap \widetilde{I^{n+1}}I=x_1(\widetilde{I^{n+1}}I :_{I} x_1)=x_1\widetilde{I^{n}} I$. Hence we have
\begin{equation}\label{eq4}
\sum_{n \ge 0}\lambda(\widetilde{I^{n+1}} I/ J \widetilde{I^n} I) =\sum_{n\geq 0}^{p-2} \lambda(I^{n+1}I/JI^nI)+(s-1)-(p-2).
\end{equation}

Let $W_J=\{t\in \mathbb{N}\,\mid J\widetilde{I}^nI\cap I^{n+1}I=JI^{n}I, \,0\leq n\leq t\}$. Then $p-2\in W_J$. Hence, by \cite[Theorem~4.2]{RV}, we have
$$r(I)\leq \sum_{n\geq 0}\lambda(\widetilde{I^{n+1}} I/ J \widetilde{I^n} I)+p-1-\sum_{n=0}^{p-2}\lambda (I^{n+1}I/JI^{n}I)=s.$$

Finally, since   ${\rm depth}\,(R/I)>0$ and 
$
0\rightarrow R/I \rightarrow {\rm gr}_I(R) \rightarrow {\rm gr}_I(R)_+ \rightarrow 0
$ is exact, by part (a), we have $${\rm depth} ({\rm gr}_I(R))\geq {\rm min} \{{\rm depth} \, R/I , {\rm depth} ({\rm gr}_I(R)_+) \} \geq 1.$$
\QED
\bigskip

We can now prove our second main result.

\begin{Theorem}\label{2}$[$Sally's Conjecture for $j$-stretched ideals$]$
 Assume $R$ and $I$ satisfy  Setting \ref{ass} (1). Let  $I$  be  $j$-stretched. If there exists a positive integer $p$ such that
  \begin{itemize}
\item[(a)] $\lambda(J\cap I^{n+1}/JI^n)=0$\, for every $0\leq n\leq p-1$, 
\item[(b)]  $\lambda(I^{p+1}/JI^p)\leq 1$,
\end{itemize}
then
\begin{itemize}
\item[(i)] for a general $x_1\in I$, \,$x_1^*$ is regular  on ${\rm gr}_I(R)_+$.
\item[(ii)] ${\rm depth}\,( {\rm gr}_I(R) )\geq d-1$.
\end{itemize}
\end{Theorem}

\demo    We prove the
theorem by induction on $d$. The case  $d=2$ has been  proved in Theorem
\ref{dim2}. Let $d\geq 3$ and assume the theorem holds for $d-1$. We
first reduce to the case of ${\rm grade}\,I  \geq 1$. If ${\rm
grade}\,I=0$,  let $H_0=0: I$. As in the proof of Theorem \ref{CM}, all assumptions still hold for the quotient ring $R/H_0$.
Furthermore,  $I/H_0 \cap I=I$, ${\rm grade} (I/H_0\cap I) \geq 1$ and
${\rm depth}\,( {\rm gr}_I(R) )\geq {\rm depth}\,( {\rm
gr}_{I}(R/H_0) )$. So we are reduced to the case where the ideal $I$
contains at least one regular element on $R$.  Thus\,  $x_1$ is regular on $R$.

If  ${\rm dim}\,R/I=0$ then the assertion  follows from
\cite[Theorem 4.4]{RV}. Hence, we may assume ${\rm dim}\,R/I\,>\,0$. Let
$^{\prime}$  denote images in $R^{\prime}=R/(x_1)$.   Observe
that $R^{\prime}$ is a Cohen-Macaulay ring of dimension $d-1$
and $\ell(I^{\prime})=d-1$. Also, $I^{\prime}$ satisfies $G_{d-1}$ and
$AN^-_{d-3}$ (see \cite[Lemma~3.2]{PX1}).
Furthermore, observe that $R^{\prime}/I^{\prime} \cong R/I $, whence
${\rm depth}\,(R^{\prime}/I^{\prime})={\rm depth}\,(R/I)\geq
{\rm min}\{{\rm dim} \, R/I, 1\}= \{{\rm dim} \,
R^{\prime}/I^{\prime}, 1\}$. Clearly, $I^{\prime}$ is
$j$-stretched in $R^{\prime}$. By induction hypothesis, for a general $x_2\in I$, $x_2^*$ is regular  on ${\rm gr}_I^{\prime}(R^{\prime})_+$, and ${\rm
depth}\,( {\rm gr}_{I^{\prime}}(R^{\prime}) )\geq d-2.$

 By \cite[Lemmas~4.8 and 4.9]{PX1}, one has  that  $x_1^*$ is regular on ${\rm gr}_I(R)$. Since ${\rm depth}({\rm gr}_{I^{\prime}}(R^{\prime}))\geq d-2$ and $x_1^*$ is regular on ${\rm gr}_I(R)$, we have ${\rm depth}({\rm gr}_{I}(R))\geq d-1$. \QED
\bigskip

As an application of Theorem \ref{2}, we  obtain a sufficient condition for the almost Cohen-Macaulayness of the associated graded rings of $j$-stretched ideals.

\begin{Corollary}\label{3}
Let  $R$ and $I$ be as in Setting \ref{ass}. If $I$ is $j$-stretched with index of nilpotency $K$, then
 \begin{itemize}
\item[(a)]   $I^{K+1}\subseteq JI^{K-1}$ if and only if $\lambda(I^{K}/JI^{K-1})=1$.
\item[(b)]  If $I^{K+1}\subseteq JI^{K-1}$  then ${\rm depth}({\rm gr}_I(R))\geq d-1.$
\end{itemize}
\end{Corollary}

\demo
Part (a) follows by the same argument as in \cite[Proposition~3.1]{RV3}.  From part (a) and Proposition \ref{VV}, one has
 $\lambda(I^{K}/JI^{K-1})\leq 1$ and  $\lambda(J\cap I^{n+1}/JI^n)=0$  for every $0\leq n\leq K-1$. Hence 
 part (b) follows by applying Theorem \ref{2} with $p=K-1$.
\QED
\bigskip

As a special case of Corollary \ref{3}, we  recover also the second main result of Polini-Xie \cite{PX1}. 
\begin{Corollary}$($\cite[Theorem~4.10]{PX1}$)$\label{PX2}
Let $R$ be a $d$-dimensional Cohen-Macaulay local ring and let $I$ be an ideal with $\ell(I) = d$, ${\rm depth}\, (R/I) \geq {\rm min}\{{\rm dim}\,(R/I),1\}$ and $I$ satisfies $G_d$ and $AN^-_{d-2}$. If $I$ has almost minimal $j$-multiplicity, then ${\rm depth}\,({\rm gr}_I(R)) \geq d-1.$
\end{Corollary}

\demo
If $I$ has almost minimal $j$-multiplicity then $K=2$. Since $I^2\nsubseteq J$ and $\lambda(I^2/IJ)=1$, one has $I^2\cap J=IJ$. Therefore,  $I^3\subseteq JI$. Now Corollary \ref{3} finishes the proof with $K=2$.
\QED
\bigskip

In \cite{R1} and \cite{RV3}, it was introduced the concept of type of an ideal $I$ with respect to a given minimal reduction $J$ of $I$. This was defined as $\tau(I)=\lambda((J:I)\cap I/J)$, a number that depends heavily on the choice of $J$.
Here we introduce a slight variation of this concept that fits with our setting. For a {\it general} minimal reduction $J$ of $I$, we set $$\tau(I)=\lambda((J:I)\cap I/J),$$ and call it the {\it general Cohen-Macaulay type} of $I$.
It follows immediately from the Specialization Lemma (\ref{specializ}) that, in presence of the $G_d$ condition, this number is 
well-defined, because it is constant for $J$ general.

\begin{Lemma}\label{type}
Assume $R$ is Cohen-Macaulay.
Let $I$ be an ideal having  $\ell(I)=d$ and the $G_d$ condition.
Then the number $\tau(I)$  is independent of the general minimal reduction $J$.
\end{Lemma}

In the same spirit of the definitions given in \cite{PX1}, we say that an ideal $I$ has {\em almost almost minimal} $j$-multiplicity if $\lambda(\overline{I^2}/x_d\overline{I})\leq 2$, or equivalently, if $K\leq 3$.

Next we want to prove that the associated graded rings of $j$-stretched ideals 
having 
almost-almost minimal $j$-multiplicity (i.e. $K=3$) and 
small general Cohen-Macaulay type are almost Cohen-Macaulay. This provides a  higher dimensional version of results of \cite{RV3}. The first step in this direction consists in proving that $j$-stretched ideals of small general type satisfy the inclusion  $I^{K+1}\subseteq JI^2$. Recall the embedding codimension of $I$ is defined as 
$h(I)=\lambda(\overline{I}/\overline{I}^2)-\lambda(\overline{R}/\overline{I})$. 

\begin{Theorem}\label{SmallType}
Assume $R$ and $I$ satisfy Setting \ref{ass} (1) and (2). Let  $I$ be  $j$-stretched with $K=s(I)$. 
Let $J=(x_1, \ldots, x_d)$ be a general minimal reduction of $I$ and set $\overline{R}=R/J_{d-1}:I^{\infty}$, 
where $J_{d-1}=(x_1, \ldots, x_{d-1})$. 
 If $\tau(I)< h(I)+1-\lambda(\overline{R}/\overline{I})$,  then
$$  \nu_2=K-2,\,\,\,  J\cap I^3=JI^2. $$
In particular, $I^{K+1}\subseteq JI^2$.
\end{Theorem}

\demo
Similar to the proof of \cite[Theorem~2.7]{RV3}.
\QED
\bigskip

The next theorem generalizes several classical results, see for instance \cite{S3}, \cite{RV2}, \cite{RV3}, \cite{RV} and \cite{PX1}.
\begin{Corollary}\label{almalm}
Assume $R$ and $I$ satisfy Setting \ref{ass} (1) and (2). Let  $I$ be $j$-stretched with $K=s(I)$ and let $h$ be the embedding codimension of $I$ as in Theorem \ref{SmallType}. If either $($i$)$ $K\leq 2$, or $($ii$)$ $K=3$ and $\tau(I)< h(I)+1-\lambda(\overline{R}/\overline{I})$, then $${\rm depth}({\rm gr}_I(R))\geq d-1.$$
\end{Corollary}

\demo  Since the cases $K=1,2$ have been proved in \cite{PX1}, we only need to prove the case $K=3$. By Theorem \ref{SmallType}, we have that $I^4\subseteq JI^2$. Corollary \ref{3} now finishes the proof.
\QED
\bigskip

We conclude this section with the example of an ideal $I$ having minimal $j$-multiplicity, not having $G_{d}$ condition and for which ${\rm gr}_I(R)$ is not Cohen-Macaulay. It demonstrates that the residual assumptions in our main Theorems are necessary.

\begin{Example} $($see \cite{CPUX} or \cite[3.10]{PX1}$)$
 Let $R=k\llbracket{x,y,z}\rrbracket/(x^3-x^2y)$ and $I=(xy^t,z)$ for any $t\ge 0$.  Then 
 $R$ is a two-dimensional  Cohen-Macaulay local ring, $\ell(I)=2$, and $I$ has reduction number zero. In particular $I$ has minimal $j$-multiplicity. However,  $I$ does not satisfy $G_2$, and ${\rm gr}_I(R)$ is not Cohen-Macaulay.
\end{Example}

\section{The $\m$-primary case}

In this  section we prove the non trivial fact that $j$-stretched ideals (strictly) generalize the stretched $\m$-primary ideals introduced by Sally and  Rossi-Valla. First, recall that an $\m$-primary ideal $I$ is said to be  {\em stretched} if there exists a minimal reduction $H$ of $I$ such that
\begin{itemize}
\item[(a)]  $H\cap I^2=HI$.
\item[(b)]  $HF_{I/H}(2)\leq 1$.
\end{itemize}

This definition, first given in \cite{RV3}, extends the classical concept of stretched Cohen-Macaulay local rings 
 given by Sally in \cite{S3}.
If $R$ is Cohen-Macaulay, stretched $\m$-primary ideals include ideals having minimal multiplicity (see for instance \cite{RV}).
However, there are $\m$-primary ideals with almost minimal multiplicity that are not stretched, even in $1$-dimensional Cohen-Macaulay local rings.
 In contrast, $j$-stretched ideals include ideals having minimal or  almost minimal multiplicity, because they include ideals of minimal and almost minimal $j$-multiplicity.
\medskip

We first prove that  general minimal reductions always achieve the minimal colength.
\begin{Proposition}\label{generic1}
Let $I$ be an ideal which has  $\ell(I)=d$ and the $G_d$ condition.  Let $H$ and $J$ be a minimal and a general minimal reduction of $I$, respectively.  Let  $n\geq 1$ be a fixed integer. Then the lengths $\lambda(I^n/J^n)$ and $\lambda(I^n/JI^{n-1}+I^{n+1})$ do not depend on $J$. Furthermore, if $R$ is  equicharacteristic, one has
\begin{itemize}
\item[(a)] $\lambda(I^n/J^n)\leq \lambda(I^n/H^n)$.
\item[(b)] $\lambda(I^n/JI^{n-1}+I^{n+1})\leq \lambda(I^n/HI^{n-1}+I^{n+1})$.
\end{itemize}
\end{Proposition}

\demo Let $\m$ be the maximal ideal of $R$  and write $I=(a_1, \ldots, a_s)$.
To  prove assertion (a),   take $d \times s$ variables, say $\underline{z}=(z_{ij})$, and set $S=R[\underline{z}]$, $J^{\prime}=(x_1^{\prime}, \ldots x_{d}^{\prime})S$, where $x_i^{\prime}=\sum_{j=1}^s z_{ij}a_j$, $1\leq i \leq d$.  Let  $\underline{\alpha_0}\in R^{ds}$ be the vector such that $J^{\prime}_{\underline{\alpha_0}}=H$.
Since $I$ has the  $G_d$ condition, we have $\lambda_{S_{\m S}}(IS_{\m S}/J^{\prime}S_{\m S})<\infty$.
By Lemma \ref{specializ}, for a general element $\underline{\alpha}\in R^{ds}$, we have
$$\lambda(I^n/J^n)=\lambda(I^n/[(J^{\prime})^nS]_{\underline{\alpha}})=
\lambda_{S_{\m S}}(I^nS_{\m S}/(J^{\prime})^nS_{\m S}).$$
Furthermore, if $R$ is  equicharacteristic, we have $\lambda(I^n/J^n)\leq
\lambda(I^n/[(J^{\prime})^nS]_{\underline{\alpha_0}})=\lambda(I^n/H^n).$
Assertion (b) can be proved similarly.
\QED
\bigskip

We can then compare the lengths of quotients that are relevant for stretched ideals.
\begin{Proposition}\label{inters}
Let $(R, \m)$ be a $d$-dimensional equicharacteristic Cohen-Macaulay local ring  with infinite residue field. Let $I$ be an $\m$-primary ideal, and $H$  a minimal reduction of $I$. Then for a general minimal reduction $J$ of $I$, one has
$$\lambda((J \cap I^2)/JI)\leq \lambda((H\cap I^2 )/HI).$$
In particular,\, if $H\cap I^2 =HI$ then one has $J \cap I^2 =JI$.
\end{Proposition}

\demo For any ideal $L\subseteq I$, we have
$\lambda((L \cap I^2 )/LI)=\lambda(I^2/LI)-\lambda(I^2/L\cap I^2)$, and
$\lambda(I^2/L\cap I^2 )=\lambda(R/L)-\lambda(R/I^2+L)$, so that we obtain
\begin{equation}\label{intersez}
\lambda((L\cap I^2 )/LI)=\lambda(I^2/LI)-\lambda(R/L)+\lambda(R/(I^2+L)).
\end{equation}
Observe that \begin{itemize}
\item $\lambda(I^2/JI)=\lambda(I^2/HI)$ (see, for instance, \cite[Corollary~2.1]{RV}).
\item $\lambda(R/J)=e(R)=\lambda(R/H)$, because $J$ and $H$ are minimal reductions of $I$.
\item $\lambda(R/I^2+J)\leq \lambda(R/I^2+H)$, by Lemma \ref{specializ}.
\end{itemize}
Together with equation (\ref{intersez}), the above gives $\lambda((J\cap I^2)/JI)\leq \lambda((H\cap I^2)/HI).$
\QED
\bigskip

We now prove the main result of this section. It shows that, in certain situations, $j$-stretchedness can be checked by using a {\em special} minimal reduction (instead of  every {\it general} minimal reduction). In particular, it gives a concrete criterion to construct examples of $j$-stretched ideals.

\begin{Theorem}\label{AN}
Let $(R, \m)$ be a $d$-dimensional equicharacteristic Cohen-Macaulay  local ring  with infinite residue field.   let $I$ be an ideal with $\ell(I)=d$. Let $H=(y_1,\dots,y_d)$ be a minimal reduction of $I$. Set $H_{d-1}=(y_1,\ldots,y_{d-1})$ and assume
$$\lambda(I^2/[y_dI + I^3 + (H_{d-1}: I^{\infty}) \cap I^2])\leq 1.$$
If one of the  following two conditions holds,
\begin{itemize}
\item[(i)]  $I$ is $\m$-primary  and\, $H\cap I^2=HI,$
\item[(ii)] ${\rm depth}\,(R/I)\geq 1$,  $I$ has properties $G_d$ and $AN^-_{d-2}$, and $H_{d-1}: I$ is a geometric $d-1$-residual intersection of $I$,
\end{itemize}
then $I$ is $j$-stretched.
\end{Theorem}

\demo It suffices to show that either (i) or (ii) implies the equality $$y_dI + I^3 + (H_{d-1}: I^{\infty}) \cap I^2=HI+ I^3.$$ Note that, to prove the above, one does not need the inequality $\lambda(I^2/[y_dI + I^3 + (H_{d-1}: I^{\infty}) \cap I^2])\leq 1$.

First assume (i) holds. Since $R$ is Cohen-Macaulay, then $I$ contains a non zero divisor on $R/H_{d-1}$, whence
$H_{d-1}: I^{\infty} = H_{d-1}$. One then has
$$(H_{d-1}: I^{\infty}) \cap I^2 = H_{d-1} \cap I^2 = H_{d-1} \cap H \cap I^2 = H_{d-1} \cap HI.$$
Now, we have
$$
\begin{array}{lll}
H_{d-1} \cap HI & = H_{d-1} \cap (H_{d-1}I + y_dI) & = H_{d-1}I + (H_{d-1} \cap y_dI)\\
                                   & = H_{d-1}I + y_d(H_{d-1} :_{I} y_d) & = H_{d-1}I  + y_dH_{d-1}\\
                                   & = H_{d-1}I,&
\end{array}
$$
showing that $(H_{d-1}: I^{\infty}) \cap I^2 = H_{d-1}I$, which immediately implies $y_dI + I^3 + (H_{d-1}: I^{\infty}) \cap I^2=HI+ I^3.$

Next, assume (ii) holds. Since $H_{d-1}: I$ is a geometric $d-1$-residual intersection of $I$ and $I$ satisfies $AN^-_{d-2}$, we have $H_{d-1}:I^{\infty}=H_{d-1}:I$. This time the equality $(H_{d-1}: I^{\infty}) \cap I^2 = H_{d-1}I$ follows by  an argument similar to \cite[Lemma~3.2]{PX1}.

Then, in either case, one has  $\lambda(I^2/HI + I^3)\leq 1$. By Lemma \ref{specializ}, this implies $\lambda(I^2/JI + I^3)\leq 1$ for a general minimal reduction $J$ of $I$, showing that $I$ is $j$-stretched.
\QED
\bigskip

As a consequence, we immediately obtain that  every stretched $\m$-primary ideal is $j$-stretched.
\begin{Corollary}\label{mprim}
Let $(R,\m)$ be an equicharacteristic Cohen-Macaulay local ring with infinite residue field and $I$ an $R$-ideal. If $I$ is a stretched $\m$-primary ideal then $I$ is a $j$-stretched ideal.
\end{Corollary}
\begin{proof}
By the first half of the definition of stretched $\m$-primary ideals $I$, these ideals satisfy assumptions (i) of Theorem \ref{AN}. As in the proof of Theorem \ref{AN}, we then have the equality
$$[y_dI + I^3 + (H_{d-1}: I^{\infty}) \cap I^2]=HI+I^3$$
In particular, we have $\lambda(I^2/[y_dI + I^3 + (H_{d-1}: I^{\infty}) \cap I^2])=\lambda(I^2/HI + I^3)=\lambda(I^2/H\cap I^2 + I^3)=HF_{I/H}(2)$.
By assumption of stretchedness, this length is at most $1$, proving that stretched ideals satisfy the inequality required in Theorem \ref{AN}. We can then apply Theorem \ref{AN} to conclude that $I$ is $j$-stretched.
\end{proof}

One may wonder if, in the $\m$-primary case, $j$-stretchedness coincides with stretchedness. In general, this is not the case. For instance, the ideal $I=(t^3,t^4)$ in the ring $A=k[\![t^3,t^{4},t^{5}]\!]$
  is $j$-stretched, has almost minimal multiplicity, but is not stretched. 
 This shows that $j$-stretched $\m$-primary ideals strictly generalize classical stretched ideals,
  even  in the $1$-dimensional case.

In contrast, we now provide a condition ensuring that  stretchedness coincides with $j$-stretchedness.  
\begin{Proposition}\label{equiv}
Let $(R, \m)$ be an equicharacteristic Cohen-Macaulay local ring and $I$ an $\m$-primary ideal. Assume $I^2\cap H=HI$ for a minimal reduction $H$ of $I$. Then $I$ is stretched if and only if $I$ is $j$-stretched.
\end{Proposition}

\demo  By Corollary \ref{mprim} we only need to show that, if $I$ is $j$-stretched, then $I$ is stretched. Let $J=(x_1,\dots,x_d)$ be a general minimal reduction of $I$ and $J_{d-1}=(x_1,\dots,x_{d-1})$. 
By $j$-stretchedness we have
 $\lambda(\overline{I}^2/(x_d\overline{I}+\overline{I}^3))\leq 1$,  where $\overline{R}=R/(J_{d-1}:I^{\infty})=R/J_{d-1}$.
By Proposition \ref{inters}, one obtains $J\cap I^2=JI$. Hence, we get
$$\overline{I}^2/(x_d\overline{I}+\overline{I}^3)\cong I^2/(JI+I^3+J_{d-1}\cap I^2)= I^2/(JI+I^3)=I^2/((J\cap I^2)+I^3).$$
Therefore, for a general minimal reduction $J$ of $I$, one has $HF_{I/J}(2)\leq 1.$ This fact, together with $J\cap I^2=JI,$ proves the stretchedness of $I$.
\QED
\bigskip

We conclude this section with an application of the above results to answer a question of Sally. If $I$ is $\m$-primary, classical examples by Sally and Rossi-Valla show that $I$ can be stretched with respect to a minimal reduction $J_1$ but not stretched with respect to a different minimal reduction $J_2$. Hence it is well known that the stretchedness property depends upon the minimal reduction. Sally \cite{S3} raised the following question: {\em To what extent does the concept of ``stretchedness'' depend upon the choice of the minimal reduction?}
We are now able to answer this question.
\begin{Corollary}\label{Sally}
Let $(R,\m)$ be an equicharacteristic Cohen-Macaulay local ring and $I$ an $\m$-primary ideal. If $I$ is stretched with respect to a minimal reduction $H$, then $I$ is stretched with respect to any general minimal reduction.
\end{Corollary}

\demo Let $J$ be a general minimal reduction of $I$. By Proposition \ref{inters}, the ``intersection property'' $J\cap I^2=JI$ follows at once from $H\cap I^2=HI$.
Then we only need to show that $\lambda(I^2/(J\cap I^2) + I^3)\leq 1$. By Proposition \ref{generic1} we have $\lambda(I^2/JI+I^3)\leq \lambda(I^2/HI+I^3)$. We have then obtained the following chain of inequalities
$$\lambda(I^2/(J\cap I^2) + I^3)=\lambda(I^2/JI+I^3)\leq \lambda(I^2/HI+I^3)=\lambda(I^2/(H\cap I^2)+I^3)\leq 1.$$
\QED
\bigskip

The examples of Sally and Rossi-Valla show that $I$ is stretched with respect to a minimal reduction does not imply that $I$ is stretched with respect to every minimal reduction of $I$. However, Corollary \ref{Sally} proves that the next best possible scenario holds, that is, $I$ is stretched with respect to a  Zariski dense open subset of minimal reductions of $I$.
\medskip

{\bf Acknowledgments.} We would like to thank Ulrich and Polini for several insightful remarks and suggestions on the material of this paper.


\begin{thebibliography}{99}


\bibitem{AGH}{ I. Aberbach, L. Ghezzi and T. H\'a, The depth of the associated graded ring of ideals with any reduction number, J. Algebra {\bf 276} (2004), no. 1, 168--179. }

\bibitem{AB}{S. Abhyankar, Local rings of high embedding dimension,
Amer. J. Math. {\bf 89} (1967), 1073--1077. }

\bibitem{AM}{ R. Achilles and M. Manaresi, Multiplicity for ideals of maximal analytic spread and intersection theory,
J. Math. Kyoto Univ. {\bf 33-4} (1993), 1029--1046. }

\bibitem{CPUX}{G. Colom$\acute{{\rm e}}$-Nin, C. Polini, B. Ulrich, and Y. Xie, Generalized Hilbert coefficients and normalization of ideals, in preparation.}


\bibitem{F}{ L. Fouli, A study on the core of ideals,  Ph.D. thesis, Purdue University, 2006.}

\bibitem{Gh}{ L. Ghezzi, On the depth of the associated graded ring of an ideal, J. Algebra {\bf 248} (2002), 688--707.}

\bibitem{GH}{ S. Goto and S. Huckaba, On graded rings associated to analytic deviation one ideals, Amer. J. Math.
{\bf 116} (1994), 905--919.}

\bibitem{GY1}{S. Goto and Y. Nakamura, On the Gorensteinness of graded rings associated to ideals of analytic deviation one, Contemp. Math. {\bf 159} (1994), 51--72.}

\bibitem{GYN}{ S. Goto, Y. Nakamura and K. Nishida, Cohen-Macaulay graded rings associated to ideals, Amer. J.
Math. {\bf 118} (1996), 1197--1213.}

\bibitem{M2} { D. R. Grayson and M. E. Stillman,  Macaulay2, a software system for research
                   in algebraic geometry, Available at http://www.math.uiuc.edu/Macaulay2.}

\bibitem{HH1}{ S. Huckaba and C. Huneke, Powers of ideals having small analytic deviation, Amer. J. Math. {\bf 114} (1992),
367-403.}

\bibitem{Hu}{ C. Huneke, Linkage and Koszul homology of ideals, Amer. J. Math. {\bf 104} (1982), 1043--1062.}

\bibitem{Hu2}{ C. Huneke, Strongly Cohen-Macaulay schemes and residual intersections, Trans. Amer. Math. Soc. {\bf 277} (1983), 739--763.}

\bibitem{JU}{ M. Johnson and B. Ulrich, Artin--Nagata properties and
Cohen--Macaulay associated graded rings, Compositio Math. {\bf
103} (1996), 7--29. }

\bibitem{JK}{ B. Johnston and D. Katz, Castelnuovo regularity and graded rings associated to an ideal, Proc. Amer. Math. Soc. {\bf
123} (1995), 727--734. }


\bibitem{NT}{ D. V.  Nhi and N. V. Trung, Specialization of modules, Comm. Algebra  {\bf 27} (1999), 2959--2978.}

\bibitem{NU}{ K. Nishida and B. Ulrich, Computing $j$-multiplicities, J. Pure Appl. Algebra {\bf 214} (2010), 2101--2110.}

\bibitem{PX1}{ C. Polini and Y. Xie, $j$-multiplicity and depth of associated graded modules, J. Algebra {\bf 379} (2013),  31--49.}

\bibitem{PX2}{ C. Polini and Y. Xie, Generalized Hilbert functions, to appera in Comm. in Algebra.}


\bibitem{R1}{ M. E. Rossi, Primary ideals with good associated graded
rings, J. Pure Appl. Algebra {\bf 145} (2000), 75--90.}


\bibitem{RV1}{ M. E. Rossi and G. Valla, A conjecture of J. Sally, Comm. in Algebra
{\bf 24 (13)} (1996),  4249--4261.}

\bibitem{RV2}{ M. E. Rossi and G. Valla, Cohen-Macaulay local rings of embedding dimension $e+d-3$, J. London Math. Soc. {\bf 80} (2000), 107--126.}

\bibitem{RV3}{ M. E. Rossi and G. Valla, Stretched $\m$-primary ideals,  Beitr$\ddot{{\rm a}}$ge Algebra Geom. {\bf 42 } (2001), 103--122.}


\bibitem{RV}{M. E. Rossi and G. Valla, Hilbert functions of filtered modules, Lecture Notes of the Unione Matematica Italiana {\bf 9}, Springer-Verlag, Berlin, UMI, Bologna, 2010.}


\bibitem{S1}{ J. Sally, On the associated graded ring of a local Cohen-Macaulay ring,
J. Math. Kyoto Univ. {\bf 17} (1977),   19--21.}

\bibitem{S3}{ J. Sally, Stretched Gorenstein rings, J. London Math. Soc. {\bf 20} (1979), 19--26.}

\bibitem{S4}{ J. Sally, Cohen-Macaulay local rings of embedding dimension $e+d-2$,
J. Algebra {\bf 83} (1983), 393--408.}

\bibitem{T}{ Z. Tang, Rees rings and associated graded rings of ideals having higher analytic deviation,
Comm. Algebra {\bf 22} (1994), 4855--4898.}

\bibitem{T2}{ N. V. Trung, Constructive characterization of the reduction numbers,
 Compositio Math.  {\bf 137}  (2003),  99--113.}

\bibitem{U}{ B. Ulrich, Artin-Nagata properties and reductions of ideals,
Contemp. Math. {\bf 159} (1994), 373--400. }


\bibitem{VV}{ P. Valabrega and G. Valla, Form rings and regular sequences,
Nagoya Math. J. {\bf 72} (1978), 91--101.}

\bibitem{W}{ H. J. Wang, On Cohen-Macaulay local rings with embedding dimension
e+d-2, J. Algebra {\bf 190} (1997),  226--240.}

\bibitem{X}{ Y. Xie, Formulas for the multiplicity of graded algebras, Trans. Amer. Math. Soc. {\bf 364}  (2012), 4085--4106.}

\end{thebibliography}
\end{document}